\documentclass[a4 paper, 12pt,final]{article}
\usepackage{geometry}
\usepackage{amsmath}
\usepackage{latexsym}
\usepackage{amsfonts}
\usepackage{amssymb}
\usepackage{amsthm}
\usepackage{amsbsy}
\usepackage{graphicx}
\usepackage{bbm}
\usepackage{color}
\usepackage{epstopdf}
\usepackage{showlabels}
\usepackage{fullpage}
\usepackage{hyperref}
\usepackage{mathrsfs}
\numberwithin{equation}{section}
\usepackage{float}

\usepackage{tikz}

\newcommand{\R}{\mathbb{R}}
\newcommand{\N}{\mathbb{N}}
\newcommand{\Z}{\mathbb{Z}}

\newcommand{\cA}{\mathcal{A}}
\newcommand{\cP}{\mathcal{P}}

\newcommand{\cQ}{\mathcal{Q}}

\newcommand{\eps}{\varepsilon}
\renewcommand{\div}{\operatorname{div}}

\newcommand{\loc}{\mathrm{loc}}

\newcommand{\Qper}{{Q\text{-per}}}					
\newcommand{\tQper}{{Q\text{-per}}}

\newtheorem{theorem}{Theorem}[section]
\newtheorem{lemma}[theorem]{Lemma}
\newtheorem{proposition}[theorem]{Proposition}

\newtheorem{remark}[theorem]{Remark}
\newtheorem{example}[theorem]{Example}

\begin{document}
	\title{\Large \bf Homogenization of magnetoelastic materials with rigid magnetic inclusions at small strains
	}

	\author{
		Raffaele Grande 
		\thanks{The Czech Academy of Sciences, Institute of Information Theory and Automation, Prague, Czech Republic} \and
		Stefan Kr\"{o}mer
		\thanks{The Czech Academy of Sciences, Institute of Information Theory and Automation, Prague, Czech Republic} \and
		Martin Kru\v{z}\'{i}k
		\thanks{The Czech Academy of Sciences, Institute of Information Theory and Automation, Prague, Czech Republic} \and
		Giuseppe Tomassetti
		\thanks{Università di Roma Tre, Rome, Italy}
	}

	\maketitle
	\begin{abstract}
		We investigate a homogenization problem  for  a linearly elastic magnetic material that incorporates  elastically  rigid magnetic inclusions firmly bonded to the matrix. 
		By considering a periodic arrangement of this material, we identify an effective magnetoelastic energy, obtained by homogenization when the period approaches zero. 
		 For comparison, we also briefly discuss alternative, essentially equivalent magnetic models naturally linked by 		a Legendre-Fenchel transform of magnetic energy density where the elastic deformation enters as a parameter. 
	\end{abstract}

\section{Introduction}
Magnetoelasticity is a physical phenomenon characterized by the mutual interaction between magnetic fields and mechanical deformation in certain materials. When a magnetic field is applied to a magnetoelastic material, it can induce mechanical strain, causing the material to change shape or dimensions. Conversely, mechanical deformation of the material, such as stretching, compressing, or bending, can alter its magnetic properties \cite{Dorfmann2014}. Magnetoelastic coupling has been harnessed in a range of advanced applications, from sensors \cite{Calkins2007}, actuators \cite{Apicella2019}, and energy harvesting devices \cite{Deng2017}. 

The magnetoelastic effect was first reported by J. P. Joule in 1842, when he observed that certain ferromagnetic materials, such as iron and nickel, undergo changes in shape or dimensions when exposed to a magnetic field \cite{Joule1847}. This effect, now known as magnetostriction \cite{Lacheisserie2002}, arises due to the influence of the magnetic field on the alignment of magnetic domains within the material, leading to either expansion or contraction. While magnetostriction is observed in many ferromagnetic materials, it is particularly pronounced in certain rare-earth iron compounds, such as Terfenol-D and Galfenol, which exhibit significantly higher magnetostrictive strains \cite{Behera2022}. These advancements have spurred significant mathematical research to understand the role of complex microstructural arrangements in these giant magnetostrictive materials \cite{James1993,Kinderlehrer1996,DeSimone1997,DeSimone2002}, primarily grounded in Brown's micromagnetic model \cite{Brown1963,Brown1966}. 

As an alternative to metal-based magnetostriction, it was observed in \cite{Rigbi1983}  that elastomers filled with ferromagnetic inclusions exhibit magnetoelastic behavior. By their similarity with magnetorheological fluids \cite{Rossi2018}, these composite materials have been referred to as \emph{magnetorheological elastomers} \cite{Ginder1999}.  A micromechanical model describing the shear deformation response of these composites was proposed in \cite{Jolly1996}. Additionally, a microscopic model in the context of statics and the geometrically-linear setting was introduced in \cite{Borcea2001}. A nonlinear macroscopic continuum model, building on previous work by \cite{Kovetz2000}, where a general formulation of the balance laws for an electromagnetic continuum was developed, was later proposed in \cite{Kankanala2004}.

The model described  in \cite{Kankanala2004} is based on the magnetization as the primary variable, and neglects exchange energy, on the basis of the fact that the scale of interest for magnetic elastomers is larger than the scale at which magnetic microstructures are observed. When exchange energy is neglected, it is possible to adopt magnetic induction or magnetic field as the primary variable on which the free energy depends, leading to several alternative formulations, both non-variational \cite{Dorfmann2003,Dorfmann2004,Steigmann2007,Ericksen2007}  and variational \cite{BuDoOg2008}. The equivalence of all these formulations has been recently addressed in \cite{sharmaVariationalPrinciplesNonlinear2021a}. To the best of our knowledge, the only mathematical work addressing the well-posedness of these formulations is contained in \cite{Silhavy2018,Silhavy2019}, where a notion of   polyconvexity  in the context of   electro/magnetoelasticity  has been introduced, and existence theorems have been established for the variational formulation based on magnetic induction as the primary variable. 

Constitutive models for magnetic elastomers have been proposed in \cite{bustamanteTransverselyIsotropicNonlinear2010a,dorfmannConstitutiveModellingMagnetosensitive2004a}.  In order to better exploit these models; it is important to understand how the microscopic structure of the materials affects their macroscopic behavior. For this reason, various microscopic micromechanical models have been proposed.  As far as theoretical work is concerned, in addition to the paper \cite{Borcea2001}, a variational homogenization framework based on the formulations proposed in \cite{Dorfmann2004,BuDoOg2008}, based on magnetic induction, has been carried out in \cite{PonteCastaneda2011}. For the same formulation, generalized convexity notions  have been proposed in \cite{furerHomogenizationMacroscopicInstabilities2022}, alongside an empirical homogenization formula inspired by the rigorous results in \cite{Muller1987}. As far as numerical homogenization is concerned, we mention the framework proposed in \cite{danasEffectiveResponseClassical2017a}, and work using the XFEM approach has been carried out in \cite{spielerXFEMModelingHomogenization2013}. 

The majority of the works cited above adopt a variational formulation where magnetic induction is the primary magnetic variable. One of the advantages of this  approach   is that  a  uniquely-defined  stationary point  is the energy minimizer. In the context of homogenization, this property is useful because it opens the way to using the tools of $\Gamma$--convergence to rigorously obtain a homogenization formula for periodic material structures.   For a more detailed comparison of models see Appendix~\ref{sec:AppA}.

In the present paper we consider a body $\Omega$ made of a soft and a rigid magnetically responsive material arranged periodically by repetition of a cell $\varepsilon Q$ of characteristic size $\varepsilon$, obtained by scaling of a reference cell $Q$ of unit size. The energy density in the reference  unit  cell $Q$ is 
\begin{equation}
	\begin{aligned}
		f(z,G,B) &:=(1-\chi_M(z))f_{soft}(G,B)+\chi_M(z) f_{rigid}(G,B),\qquad z\in  Q, \\
	\end{aligned}
\end{equation}
where $G=\nabla u$ is the displacement gradient, $B$ is the magnetic induction, and $\chi_M$ is the characteristic function of a set $M$  that defines the region occupied by the rigid material. We assume that $M$ does not touch the boundary of the unit cell. Thus, the rigid material is an inclusion contained in a soft matrix. 

To simplify the analysis and avoid technical complications, which could be resolved but would make the exposition cumbersome, we exclude the rigid inclusions from the cells that intersect the boundary. Then, the functional that governs the equilibrium configurations is given by the sum of the magneto-elastic energy:
\begin{equation}\label{functional}
\mathcal{E}_{\varepsilon}(u, B)=\int_{\Omega_{\varepsilon}} f\left(\frac{x}{\varepsilon}, \nabla u, B\right) dx+\int_{\Omega \backslash \Omega_{\varepsilon}} f_{soft}(\nabla u, B) dx+\int_{\mathbb{R}^3\setminus \Omega} \frac{1}{2 \mu_0}|B|^2 dx,
\end{equation}
and a linear perturbation describing the applied magnetic field. Here, $\Omega_\varepsilon$ denotes the union of all cells fully contained within $\Omega$, while $\Omega\setminus\Omega_\varepsilon$ represents a thin boundary layer adjacent to $\partial\Omega$. The functional depends on the mechanical displacements $u\in W^{1,2}(\Omega;\mathbb R^3)$ and magnetic induction $B\in L^2(\mathbb R^3;\mathbb R^3)$ such that $u=0$ on $\partial\Omega$ and $\operatorname{div}B=0$ in the sense of distributions.

The energy functional \eqref{functional} is an approximation to the 
small-strain regime, 
where the displacement gradient $G=\nabla u$ is a small perturbation of the deformation gradient $F=I+\nabla u$ around the identity, and the magnetic induction in the reference and current configuration can be described by the same  field $B$. The functional can be obtained by a formal approximation of the actual functional that governs equilibrium in the large-displacement, large-strain regime, which is not easy to study rigorously. Some related results have been obtained in context of micromagnetics \cite{BreKru23a,BreDaKru23a}, but homogenization seems to be presently out of reach in the large strain case even without magnetic effects. 

Considering the situation where the cell size is considerably smaller than the specimen's dimensions, we wish to capture the macroscopic attributes of this physical system in the limit as $\varepsilon$ tends to zero. We achieve this goal by employing $\Gamma$--convergence \cite{Dal12IGC} to identify the limit  functional of \eqref{functional}. This area of study has remained a dynamic realm of mathematical exploration for  decades. We refer to \cite{BakPan12} or \cite{CioDon99,Dal12IGC} for an introduction to the subject. 
In context of micromagnetics without elasticity, homogenization has been studied in \cite{Pisante}. 
Homogenization of materials featuring periodic inclusions with vastly distinct properties compared to the matrix material, called {\it high-contrast materials} has garnered significant attention, too, see  \cite{BraiGa95a, CheChe12TSGIF, CheCheNe17HCHNE, DaKrPa22}. This surge in interest stems from the remarkable versatility and applicability of such composite materials. 

The main result of this paper is that the $\Gamma$--limit of the functional $\mathcal{E}_\varepsilon$ as $\varepsilon$ tends to zero, is 
\begin{equation}\label{functional-hom}
	\mathcal{E}_{hom}(u, B)=\int_{\Omega} f_{hom}\left(\nabla u, B\right) dx+\int_{\mathbb{R}^3\setminus \Omega} \frac{1}{2 \mu_0}|B|^2 dx,
\end{equation}
 where the homogenized energy density $f_{hom}$ is given by a multi-cell formula (see \eqref{eq:deffhom} below).

		
The paper is organized as follows. In Section \ref{sec:proto} we formulate the problem and we provide our $\Gamma$--convergence result. In Section \ref{sec:examples} we offer a few examples that fit within the assumption of the convergence result. In Section \ref{sec:hom} we give the proof of the convergence result.  The Appendix discusses altenative magentic models and passage between them based on \cite{BuDoOg2008}, including some information on parameter-dependent Legendre-Fenchel transforms that are relevant for the mathematical properties of the transformed model.

\section{Problem formulation and main result}\label{sec:proto}
\subsection{Definitions}
	Let $Q=(0,1)^3$ be the unit cube and let $M\subset \R^3$ be an open, $Q$-periodic set
	representing elastically rigid inclusions of fixed size. 
	The characteristic function of $M$, which {is} denoted by $\chi_M:\R^3\to \{0,1\}$, \color{black}is therefore also $Q$-periodic. 
We require that the inclusions are well-contained in the unit cell in the sense that
  \begin{equation}
     \overline{M\cap Q}\subset Q {.} \label{Mwellsep}
  \end{equation}
Given the microscopic length scale parameter $\eps>0$
	and a Lipschitz domain $\Omega\subset\R^3$ representing the reference configuration 
	of the body under study, the rigid inclusions are described by the set
	\begin{equation}\label{Meps}
		\begin{aligned}
			M_\eps=\bigcup_{z\in \mathcal{Z}_\eps(\Omega)} \eps(z+ M\cap Q), 
			\quad\text{where}~\mathcal{Z}_\eps:=\{z\in  \Z^3\mid \eps(z+ M\cap Q)\subset \Omega \}.&
		\end{aligned}
	\end{equation}
	Any set of the form $\eps(z+ Q)$ with a $z\in \Z^3$ is called an $\eps$-cell.
	In that sense, $z\in \mathcal{Z}_\eps(\Omega)$ characterizes interior $\eps$-cells, i.e., those that are fully contained in $\Omega$. Their union is defined as
	\begin{equation}\label{Omegaeps}
		\begin{aligned}
			\Omega_\eps:=\bigcup_{z\in \mathcal{Z}_\eps(\Omega)} \eps(z+ Q).&
		\end{aligned}
	\end{equation} 
 Slightly more general than \eqref{functional},
we  consider a class of functionals of the form 
	\begin{equation}\label{Eeps}
		E_{\varepsilon}(u, B)=\int_{\Omega_{\varepsilon}} f\left(\frac{x}{\varepsilon}, \nabla u, B\right) d x+\int_{\Omega \backslash \Omega_{\varepsilon}} f_b(\nabla u, B) d x+\int_{\mathbb{R}^3\setminus \Omega} f_{e x t}(x, B) d x,
	\end{equation}
	where the pair $(u,B)$ -- displacement and magnetic induction -- belongs to the admissible set
	\begin{equation}\label{eq:UpartialOmega}
		\mathcal{U}_{\partial \Omega}:=\left\{(u, B) \in W^{1,2}\left(\Omega ; \mathbb{R}^3\right) \times L_{\text{div}}^2\left(\Omega ; \mathbb{R}^{3}\right) \mid u=0 \text { on } \partial \Omega \text { (in the sense of traces)}\right\}
	\end{equation}
where
\begin{equation}
	L_{\text {div}}^2\left(\mathbb{R}^3;\mathbb{R}^{3}\right)=\left\{B \in L^2\left(\mathbb{R}^3 ; \mathbb{R}^{3}\right) \mid \operatorname{div} B=0 \text { in the sense of distributions}\right\} .
\end{equation}		
\begin{remark}
			By requiring that inclusions only occur in $M_\eps$, we assume that whenever an $\eps$-cell intersects $\partial\Omega$, it does not contain a rigid inclusion and is instead fully filled by the elastically soft material represented by the energy density $f_b$ below. This effectively leads to a boundary layer on $\Omega\setminus M_\eps$ without inclusions, a neighborhood of $\partial\Omega$ of the order of $\eps$. This assumption avoids technical issues that would arise if the boundary $\partial\Omega$ cuts an inclusion in an $\eps$-cell into several pieces, possibly rough and not well separated from $\partial\Omega$. To handle this kind of situation, we would need a much more subtle	generalization of Lemma~\ref{lem:wlim-approx} below which is not obvious.
			However, in our proof, we will see that in any case, such an $\eps$-boundary layer of soft material is energetically negligible in the limit as $\eps\to 0$. In view of this, we rather avoid the issue altogether.
		\end{remark}

To prove our main result, we will require some general assumptions on the integrands in \eqref{Eeps}. The examples in the next section are consistent with these assumptions.

Given a matrix $G\in \R^{3\times 3}$ we denote by $\operatorname{sym}G:=\frac{1}{2}( G^\top+ G)$ 
		and $\operatorname{skw}(G):=\frac{1}{2}( G^\top- G)$, respectively, its symmetric and skew-symmetric part. We assume that the function $f$ appearing as integrand in \eqref{Eeps}, satisfies:
		\begin{subequations}
		\begin{align} 
	& \begin{aligned}
		&f:\R^3\times \R^{3\times 3}\times \R^3\to \R\cup\{+\infty\},\\
		&z\mapsto f(z, G, B)~~\text{is  $Q$-periodic and  measurable} \text{ for all }G\in\mathbb R^{3\times 3}, B\in\mathbb R^3;
	\end{aligned}\label{f0}
	\\
	&\begin{cases} 
		f(z, G, B)=f(z, \operatorname{sym}(G), B)~~\text{for all $(z,G,B)$} & \text{if $z\notin M$}, \\
		\frac{1}{C}(| \operatorname{sym}(G)|^2+|B|^2)-C \leq f(z, G, B) \leq C(| \operatorname{sym}(G)|^2+|B|^2)+C,
		& \text{if $z\notin M$};
	\end{cases}
	\label{f1}
	\\ 
	& \begin{aligned} 
		& \begin{cases} 
			\frac{1}{C}|B|^2  \leq  f(z, G, B)\leq C(|G|^2+|B|^2)+C 	& \text{if $\operatorname{sym}(G)=0$ and $z\in M$,}\\
			f(z, G, B)=+\infty & \text{if $\operatorname{sym}(G)\neq 0$ and $z\in M$;}
		\end{cases}
	\end{aligned}\label{f2}
	\\ 
	& \begin{aligned} &|f(z, G_0, B_0) - f(z, G_{1}, B_{1})| \\
		&\leq C 
		(|G_0|+| G_{1}| + | B_0|+| B_{1}|)(| G_0 -  G_{1}| + | B_0 -  B_{1}|) \quad \text{if $z\notin M$.}
	\end{aligned}\label{f3} \\
	& \begin{aligned} &|f(z, \operatorname{skw}(G_0), B_0) - f(z, \operatorname{skw}(G_{1}), B_{1})| \\
		&\leq C 
		(|\operatorname{skw}(G_0)|+|\operatorname{skw}(G_{1})| + | B_0|+| B_{1}|)(|\operatorname{skw}(G_0) -  \operatorname{skw}(G_{1})| + | B_0 -  B_{1}|) \quad \text{if $z\in M$.}
	\end{aligned}\label{f4}
\end{align}
\end{subequations}
		For the boundary layer energy density $f_b:\R^{3\times 3}\times \R^3\to \R$, $(G,B)\mapsto f(G, B)$ (independent of $z$), we will assume the 
		$f_b$ satisfies \eqref{f1} and \eqref{f3} in place of $f$ for the ``soft material'' case $z\notin M$.
		As to $f_{ext}$, 
		the associated functional $B\mapsto \int_{\R^3\setminus\Omega} f_{ext}(x,B)\,dx$ 
		is defined on $L^2_{\div}(\R^3;\R^3)$
		and we will only need that
		\begin{align} 
			&\begin{aligned}
				&\int_{\R^3\setminus\Omega} f_{ext}(x,B)\,dx\geq \frac{1}{C}\int_{\R^3\setminus \Omega} |B|^2\,dx-C~~~\text{for all }B\in L^2(\R^3;\R^3); 
			\end{aligned}	\label{fext0}\tag{$f_{ext}$:0}
			\\
			&\begin{aligned}
				B\mapsto \int_{\R^3\setminus\Omega} f_{ext}(x,B)\,dx~~~\text{is continuous on $L^2(\R^3;\R^3)$;} &
			\end{aligned}	\label{fext1}\tag{$f_{ext}$:1}
			\\
			&\begin{aligned}
				B\mapsto \int_{\R^3\setminus\Omega} f_{ext}(x,B)\,dx~~~\text{is sequentially weakly  lower semicontinuous on $L^2(\R^3;\R^3)$}.&
			\end{aligned}	\label{fext2}\tag{$f_{ext}$:2}
		\end{align}
		\begin{remark}
			For instance, \eqref{fext1} can be obtained as a consequence of a growth condition and continuity of the integrand,
			and \eqref{fext2} clearly holds if $B\mapsto \int_{\R^3\setminus\Omega} f_{ext}(x,B)\,dx$ is 
			convex and finite valued. In particular, both is the case for our prototype examples where
			$
			f_{ext}(x,B)=\frac{1} {2\mu_0} B\cdot B
			$
with some $\mu_0>0$. The assumption \eqref{fext0} of course also holds for this example.
		\end{remark} 
		\begin{remark}\label{rem:equi-coercive}
			Coercivity of $E_\eps$ follows from \eqref{fext0} and the lower bounds in \eqref{f1} and \eqref{f2}, together with Korn's inequality (see \cite{Cia1}). In fact, we even get $2$-equi-coercivity, i.e.,
			\[
			E_\eps(u,B)\geq c \|u\|^2_{W^{1,2}(\Omega;\R^3)}+c \|B\|^2_{L^{2}(\Omega;\R^3)}-\frac{1}{c}~~~\text{for all $(u,B)\in \mathcal U_{\partial\Omega}$},
			\]
			with a constant $c>0$ independent of $u$, $B$ and $\eps$.
		\end{remark}

		\subsection{Main result}
		We will see that the homogenized functional corresponding to \eqref{Eeps} is given by
		\begin{equation}\label{Ehom}
			E_{hom}(u, B)=\int_{\Omega} f_{hom}\left(\nabla u, B\right)\, dx
			+\int_{\R^3\setminus \Omega} f_{ext,hom}(x,B)\, dx.
		\end{equation}
		Here, the homogenized energy densities $f_{hom}$ and $f_{hom,ext}$ are
		\begin{align} \label{eq:deffhom}
			\begin{aligned}
				&f_{ext,hom}=f_{ext},\\
				&f_{hom}(G,B)=\inf_{k, \varphi, \beta} 
				\int_{Q}  f\Big(kz,G+\nabla \varphi(z),B+\beta(z)\Big) \,dz,
			\end{aligned}
		\end{align}
		where the infimum is taken over all 
		\begin{align} \label{eq:fhom-infclass}
			\begin{aligned}
				&k\in\mathbb{N},~~
				\varphi\in W_0^{1,2}(Q;\mathbb{R}^3),
\\
				&\beta \in L^2_\Qper(\R^3;\mathbb{R}^3)~~\text{s.t.~$\int_Q \beta(z)\, dz=0$ and $\div \beta=0$ on $\mathbb{R}^3$}.
			\end{aligned}
		\end{align} 
		Notice that in order to be a relevant competitor for the infimum in \eqref{eq:deffhom}, $\varphi$ must also satisfy 
		\[
			\text{$\operatorname{sym}(G+\nabla \varphi(z))=0$ whenever $kz\in M$,}
		\]
		because otherwise, the integral becomes infinite by \eqref{f2}.

		\begin{remark}
		The homogenized density $f_{hom}$ is here defined by a multi-cell formula (i.e., we might have to choose $k>1$ in \eqref{eq:deffhom}). However, if $f$ is convex in $(G,B)$, it can be shown that the single cell formula suffices (i.e., $k=1$ is optimal), essentially as in \cite[Lemma 4.1]{Muller1987}. For for linearly elastic energy 
		densities, convexity in $G$ (the displacement gradient) is natural, but joint convexity in $(G,B)$
		heavily depends on the elastomagnetic coupling.
		\end{remark}

		The main result of the paper is the following. \color{black}
		\begin{theorem}\label{thm:ConvLin}
			Let $\Omega\subset \R^3$ be a bounded Lipschitz domain, 
			let $Q=(0,1)^3$ the unit cube and let $M_\eps\subset \Omega_\eps\subset \Omega$ given by \eqref{Meps} and \eqref{Omegaeps}, respectively, where
			$M\subset \R^3$ is a $Q$-periodic open set satisfying \eqref{Mwellsep}. We consider the functionals
			$E_\eps$ and $E_{hom}$ be given by \eqref{Eeps} and \eqref{Ehom}, with integrands 
			$f$, $f_{ext}$ and $f_{hom}$, where $f_{hom}$ is defined in \eqref{eq:deffhom},
			$f$ satifies \eqref{f0}--\eqref{f4}, $f_b$ satisfies \eqref{f1} and \eqref{f3} (for the case $z\notin M$) and $f_{ext}$ satisfies \eqref{fext0}--\eqref{fext2}.
			Then $E_\eps$ $\Gamma$-converges to $E_{\rm hom}$ as $\eps\to 0$, i.e., 
			the following holds for each $( u, B)\in  \mathcal U_{\partial\Omega}\subset W^{1,2}(\Omega;\R^3)\times L^2_{\div}(\R^3;\R^3)$:
			\begin{enumerate}
				\item[(i)] ($\Gamma-\liminf$ inequality) For all  sequences such that 
				$( u_\eps, B_\eps)\in \mathcal U_{\partial\Omega}$  such that $( u_\eps, B_\eps)\rightharpoonup (u,B)$ weakly in 
				$W^{1,2}(\R^3;\R^3)\times L^2(\R^3;\R^3)$,
				\begin{align}\label{eq:thmLiminf}
					\liminf_{\eps\to 0} E_\eps( u_\eps, B_\eps)\geq E_{hom}( u, B);				
					\end{align}
				\item[(ii)] (recovery sequence) There exists a sequence $( u_\eps, B_\eps)\in \mathcal U_{\partial\Omega}$ 
				such that $( u_\eps, B_\eps)\rightharpoonup( u, B)$ weakly in 
				$W^{1,2}(\R^3;\R^3)\times L^2(\R^3;\R^3)$ and
				\begin{align}\label{eq:thmLimsup}
					\lim_{\eps\to 0} E_\eps( u_\eps, B_\eps)= E_{hom}( u, B).
				\end{align}
				In addition, $ B_\eps\to  B$ strongly in $L^2(\R^3\setminus\Omega;\R^3)$.
			\end{enumerate}
		\end{theorem}
		\begin{remark}
		As anticipated, the presence of a layer $\Omega\setminus\Omega_\varepsilon$ does not affect the final result. In particular, the homogenized energy $f_{hom}$ does not depend on $f_b$.
		\end{remark}\color{black}

\section{Examples}\label{sec:examples}

As an example (without additional externally applied fields), we can consider the energy densities
\begin{equation}
\begin{aligned}
    f(z,G,B) &:=(1-\chi_M(z))f_{soft}(z,G,B)+\chi_M(z) f_{rigid}(z,G,B) \\
    f_b(z,G,B) &:=f_{soft}(z,G,B), \\
    f_{ext}(x,B) &:= \chi_{\R^3\setminus \Omega}(x) f_v(B) 
\end{aligned}
\end{equation}
for $z\in Q$, $x\in \R^3$, $G\in \R^{3\times 3}$ and $B\in \R^{3}$. 
Here,
\begin{equation}\label{fvac}
\begin{aligned}
f_v(B):={}& \frac 1 {2\mu_0} |B|^2. 
\end{aligned}
\end{equation}
is the magnetic field energy density in vacuum.
The other, material-dependent energy contributions $f_{soft}$ and $f_{rigid}$ can be chosen as in any of the following examples. Below, $\operatorname{sym}G=\frac{1}{2}( G^\top+ G)$ denotes  the symmetric part of a matrix $G\in \R^{3\times 3}$.
\begin{example}[simple diamagnetic or paramagnetic soft and rigid materials] 
{Let us consider $f_{soft}$ and $f_{rigid}$ as}
\begin{equation}
\begin{aligned}
f_{soft}(z, G,B)&:={} f_v(B)+\frac 1 2 \mathbb C\operatorname{sym}(G):\operatorname{sym}(G)+\bigg(\frac{1}{2\mu_{soft}}-\frac{1}{2\mu_0}\bigg)|B|^2,\\
 f_{rigid}(z, G, B )& :=f_v(B)+\begin{cases}
     \Big(\frac{1}{2\mu_{rigid}}-\frac{1}{2\mu_0}\Big)|B|^2  &\text{if } \operatorname{sym}(G)=0,
      \\
     +\infty & \text{if }\operatorname{sym}(G)\neq 0,	
    \end{cases}
\end{aligned}
\end{equation}
with a given elasticity tensor $\mathbb C\in \R^{(3\times 3)\times (3 \times 3)}$ 
and given magnetic permeabilities $\mu_{soft}<0$ (diamagnetic if $\mu_{soft}<\mu_0$, paramagnetic if $\mu_{soft}>\mu_0$) and $\mu_{rigid}<0$. Notice that in the paramagnetic case, the presence of $f_v$ is crucial to ensure coercivity of the combined energy density $f$ in $B$.
The associated Eulerian magnetization is $M=-\partial_B (f(z,G,B)-f_v(B))$; for instance,
$M=\big(\frac{1}{\mu_0}-\frac{1}{\mu_{soft}}\big)B$ in the soft part $\Omega_\eps\setminus M_\eps$. Since we always have that $B=\mu_0(M+H)$, the latter is equivalent to the classical constitutive relation $B=\mu_{soft} H$.
However, this model lacks any {coupling} interaction between $B$ and $G$.
\end{example}
\begin{example}[magnetoactive, initially paramagnetic soft material with paramagnetic rigid inclusion] \label{ex:mel2}
With the vacuum energy density $f_v$ given by \eqref{fvac}, {let us consider $f_{soft}$ and $f_{rigid}$ as}
\begin{equation}
\begin{aligned}
f_{soft}(z, G,B)&:={} f_v(B)+ 
\frac{1}{2} \mathbb C\big(\operatorname{sym}(G)-E_0(z,B)\big):\operatorname{sym}(G)+\Big(\frac{1}{2\mu_{soft}}-\frac{1}{2\mu_0}\Big)|B|^2,\\
f_{rigid}(z, G, B )&:=f_v(B)+\begin{cases}
    \Big(\frac{1}{2\mu_{rigid}}-\frac{1}{2\mu_0}\Big)|B|^2  
			&\text{if } \operatorname{sym}(G)=0,  \\
    +\infty 
			& \text{if }\operatorname{sym}(G)\neq 0,	
    \end{cases}
\end{aligned}
\end{equation}
with a given elasticity tensor $\mathbb C\in \R^{(3\times 3)\times (3 \times 3)}$ (symmetric and positive definite), 
{the permeability of $\mu_{soft},\mu_{rigid}>\mu_0$} and a $Q$-periodic pre-strain $E_0(z,B)$ depending on $B$, for instance a symmetric matrix of the form
\[
    E_0(z,B):=
    \frac{1}{|B|}B\otimes B-\frac{|B|}{3}I,
\]
here chosen trace-free and thus representing a (linearized) incompressible prestrain. Notice that for symmetry reasons, any physically realistic $E_0$ should be invariant under sign changes of $B$.
Again, $f_{soft}$ and $f_{rigid}$ can be coercive in $B$, the former provided that $|\mathbb C|$ is small enough or $\frac{1}{\mu_{soft}}$ is big enough, but in both cases only due to the presence of $f_v$ which compensates for other contributions with the wrong sign. 
This interplay also demonstrates that it is not natural to assume that $f$ is everywhere convex in $B$.

For the zero displacement case $G=0$, we have
\[
	f_{soft}(z, 0,B)=f_v(B)+\Bigg(\frac{1}{2\mu_{rigid}}-\frac{1}{2\mu_0}\Bigg)|B|^2,
\]
which corresponds to a paramagnetic behavior with the magnetization 
\[
	M=-\frac{\partial}{\partial B} (f_{soft}(z, 0,B)-f_v(B))=
\bigg(\frac{1}{2\mu_0}-\frac{1}{2\mu_{soft}}\bigg) B,\]
cf.~\eqref{b-to-m} in Remark~\ref{rem:from-MH-to-B}. 
However, due to the coupling term between $G$ and $E_0$, the magnetic behavior becomes more complex as soon as 
the soft material deforms.
If we use the linear relation between $B$ and $M$ at $\operatorname{sym}(G)=0$ to rewrite $E_0$ in terms of $M$, 
we formally recover the linearized model obtained in \cite{ALKruMo24a} in context of micromagnetism. We miss both the exchange energy term $|\nabla M|^2$ and the saturation constraint $|M|=1$ included there, though. Moreover, to fully understand the coupling of $G$ and $M$ of our model here in the corresponding functional with magnetic state $(M,H)$ (cf.~Remark~\ref{rem:from-MH-to-B}) for general $G$,
we have to compute
its material density given by $\hat\Phi(z,G,M)= \Phi_{z,G}^*(M)-\frac{1}{2\mu_0}|M|^2$
with $\Phi_{z,G}(B):=-(f_{soft}(z, G,B)-f_v(B))$. Unfortunately, no way of obtaining an explicit form of $\Phi_{z,G}^*(M):=\sup_{B\in \R^3} (B\cdot M-   \Phi_{z,G}(M))$
of   $\Phi_{z,G}$,  the Fenchel conjugate involved here, seems to be known.  
\end{example}

\begin{example}[alternative magnetoactive soft material with paramagnetic inclusions] \label{ex:mel3}
We can also obtain an example directly based on Remark~\ref{rem:from-MH-to-B}, 
effectively initially imposing the magnetoelastic material response in terms of $(G,M)$ instead 
of $(G,B)$, and passing to the transformed densities. Take, for instance,
\begin{equation}
\begin{aligned}
f_{soft}(z, G,B)&:={} f_v(B)+ \Phi_{G}(B),\\
f_{rigid}(z, G, B )&:=f_v(B)+\begin{cases}
    \Big(\frac{1}{2\mu_{rigid}}-\frac{1}{2\mu_0}\Big)|B|^2  
			&\text{if } \operatorname{sym}(G)=0,  \\
    +\infty 
			& \text{if }\operatorname{sym}(G)\neq 0,	
    \end{cases}
\end{aligned}
\end{equation}
with the vacuum energy density $f_v$ given by \eqref{fvac} and 
\[
\begin{aligned}
	\Phi_{G}(B)&:=-\hat\Psi_{G}^*(B)=-\sup_{M\in \R^3}\big(M\cdot B-\hat\Psi_{\operatorname{sym}(G)}(M)\big),~~\text{where}\\
	\hat\Psi_{G}(M)&:=
	\frac{1}{2} \mathbb C\big(\operatorname{sym}(G)-E_0(M)\big):\big(\operatorname{sym}(G)-E_0(M)\big)
	+\alpha (|M|^2+|M|^p)+\frac{\mu_0}{2} |M|^2,	
\end{aligned}	
\]
cf.~\eqref{hatPhi-to-Phi}.
Here, $p\geq 2$, $\alpha>0$ , $\mu_{rigid}>\mu_0$ (for paramagnetic inclusions) and
$\mathbb C\in \R^{(3\times 3)\times (3\times 3)}$ (symmetric and positive definite) are parameters,
and $E_0(M)$ is a magnetically induced pre-strain, for instance of the form
\[
		E_0(M):=|M|^{\frac{p}{2}-2}M\otimes M-\frac{\beta}{3}|M|^{\frac{p}{2}}I
\]
with some $\beta\in \R$, 
similarly as in Example~\ref{ex:mel2} inspired by the model of \cite{ALKruMo24a}. 
In particular, we can choose $E_0(M)=M\otimes M-\frac{1}{3}|M|^{2}I$ for $p=4$ and $\beta=1$ (incompressible pre-strain).

One can check that for any $p\geq 2$ and $\alpha>0$, 
the assumptions of our main result, Theorem~\ref{thm:ConvLin}, hold in this example.
In fact, they are easy to see except for \eqref{f1} and \eqref{f3} which are related to $f_{soft}$.  

To prove \eqref{f1} (growth and coercivity in the soft material part), 
we can use that $\mathbb C$ is positive definite and Young's inequality to split the coupling term in $\hat\Psi_{G}(M)$, so that for any $\delta>0$, 
\[
\begin{aligned}
	c_1\Big(1-\frac{1}{\delta}\Big) |\operatorname{sym}(G)|^2
	+c_1(1-\delta) |E_0(M)|^2+\alpha (|M|^2+|M|^p) &\\
	\leq \hat\Psi_{G}(M)-\frac{\mu_0}{2} |M|^2
	\leq C_1 (|\operatorname{sym}(G)|^2+|E_0(M)|^2)&,
\end{aligned}
\]
with some constants $c_1,C_1>0$. Since $\alpha>0$ and $c_2|M|\leq |E_0(M)|\leq C_2|M|^{\frac{p}{2}}$,
we can choose $\delta:=1+\frac{\alpha}{2C_2}>1$ to see that
\[
	c_3|\operatorname{sym}(G)|^2+\frac{\alpha}{2}|M|^p+ \frac{\mu_0}{2}|M|^2
	\leq \hat\Psi_{G}(M)
	\leq C_3 (|\operatorname{sym}(G)|^2+|M|^p+1).
\]
Since $\frac{\alpha}{2}|M|^p+\frac{\mu_0}{2}|M|^2\geq \frac{\alpha+\mu_0}{2}|M|^2-\frac{\alpha}{2}$
and for any $c>0$, $(\frac{c}{2}|\cdot|^2)^*=\frac{1}{2c}|\cdot|^2$ and
$(\frac{c^p}{p}|\cdot|^p)^*=\frac{c^{-p'}}{p'}|\cdot|^{p'}$,
where $p':=\frac{p}{p-1}$ (i.e., $\frac{1}{p'}+\frac{1}{p}=1$),
this implies that with some $c_4>0$,
\[
\begin{aligned}
	&-\Big(c_3|\operatorname{sym}(G)|^2-\frac{\alpha}{2}\Big)+\frac{1}{2}\frac{1}{\alpha+\mu_0}|B|^2 \\
	&\qquad \geq
	\hat\Psi_{G}^*(B) \geq -C_3(|\operatorname{sym}(G)|^2+1)+c_4 |B|^{p'}
	\geq -C_3(|\operatorname{sym}(G)|^2+1). 
\end{aligned}
\]
As $\alpha>0$, this ensures that $f_v(B)=\frac{1}{2\mu_0}|B|^2$ dominates $-\hat\Psi_{G}^*$ concerning growth and coercivity for large values of $|B|$, and we conclude that $f_{soft}(z, G,B)$ indeed satisfies \eqref{f1}.
Assumption \eqref{f3} (Lipschitz property in the soft part) now follows from Lemma~\ref{lem:Fenchel-Lipschitz} below (see also Remark~\ref{rem:Fenchel}), with $q=\frac{p}{2}$ and $\R^k=\R^6$ identified with the symmetric matrices in $\R^{3\times 3}$, using $\R^k\ni \tilde{G}=\operatorname{sym}(G)$ as the parameter of $\Theta(\tilde{G},B):=\hat\Psi_G(B)$.
Notice that in our proof of \eqref{f1}, we have in particular shown the growth and coercivity condition \eqref{Theta-gc} assumed in part (i) of the lemma, and \eqref{Theta-gro2} can be obtained in a similar fashion.

Just as in Example~\ref{ex:mel2}, we unfortunately lack a way of computing $\hat\Psi_{G}^*$ explicitly unless $\operatorname{sym}(G)=0$.
\end{example}

\begin{remark}
Without additional external forcing or nontrivial boundary conditions, the energy in all examples is minimized by $(u,B)=(0,0)$. For example, such a forcing can be included by adding  
the energy of an externally applied magnetic field $h_{a}\in L^2(\R^3;\R^3)$ to $E_{\varepsilon}(u, B)$,
i.e., the additional term
\[
	\int_{\R^3} \frac{1}{\mu_0} b_{a}\cdot B(x)\,dx.
\]
This term just amounts to a $L^2$-weakly continuous perturbation of $E_{\varepsilon}$ which does not interfere with the homogenization, by general properties of $\Gamma$-convergence (see for example \cite{Dal12IGC}). We can therefore safely ignore it from now on. 
\end{remark}

		\section{Proof of the convergence result}\label{sec:hom}
		
		\subsection{A few preliminaries}\label{ssec:prelim}

		For the lower bound, we will face the technical problem that a weak limit $ u$ of microscopically admissible states in general does not inherit anything from the implicit microscopic rigidity constraint $\operatorname{sym}(\nabla  u_\eps)=0$ on $\{\chi_M(\cdot/\eps)=1\}$. 
		To compensate this, we will employ a suitable approximation given by the lemma below.
		Functionally, this is very similar to related arguments in \cite{BraiGa95a} handling rigid inclusions.

		\begin{lemma}\label{lem:wlim-approx}
			Let $\Lambda\subset \R^3$  be a  Lipschitz domain and 
			$w\in W^{1,2}(\Lambda;\R^3)$. In addition, assume that the measurable, $Q$-periodic characteristic function $\chi_M$ 
			satisfies $\{\chi_M=1\}\cap Q\subset\subset Q$ with the unit cube $Q$,
			and define
			\[
			\cQ(\eps):=\big\{ Q_\eps:=\eps( Z+Q)\subset \R^3 \mid  Z\in \Z^3~\text{and}~Q_\eps\subset\subset \Lambda \big\}
			\]
			and 
			\[
			M(\eps):=\big\{  x\in {\textstyle\bigcup_{Q_\eps\in \cQ(\eps)}Q_\eps}\,\big|\, \chi_M\big(\tfrac{1}{\eps}  x\big)=1 \big\}.
			\]
			Then there exists a function ${\hat w}_\eps\in W^{1,2}(\Lambda;\R^3)$ 
			(also depending on  $\Lambda$ and $\chi_M$) such that all of the following holds:
			\begin{align}\label{eq:hatu}
				\begin{aligned}
					&\hat w_\eps=w \quad\text{on $\Lambda\setminus {\textstyle\bigcup_{Q_\eps\in \cQ(\eps)}Q_\eps}$};\\
					&\hat w_\eps=w \quad\text{on $\partial Q_\eps$ for each $Q_\eps\in \cQ(\eps)$ (in the sense of traces)};\\
					& \nabla \hat w_\eps= 0\quad\text{on $M(\eps)$};\\
					& \|\nabla\hat w_\eps\|_{L^2(Q_\eps;\R^{3\times 3})}\leq C 
					 \|\nabla w\|_{L^2(Q_\eps;\R^{3\times 3})}  \quad\text{for all $Q_\eps\in \cQ(\eps)$}.
				\end{aligned}
			\end{align}
			Here, $C=C(\chi_M)\geq 1$ is a constant independent of $w$, $\eps$, $Q_\eps$ and $\Lambda$. In addition, we have that 
			\begin{align}\label{eq:hatu2}
				\begin{aligned}
					\hat w_\eps \underset{\eps\to 0}{\rightharpoonup} w\quad \text{weakly in}~~W^{1,2}(\Lambda;\R^3).
				\end{aligned}
			\end{align}
			If $w\in W^{1,\infty}(\Lambda;\R^3)$, then so is $\hat w_\eps$, and there is a constant $C>0$ 
			only depending on the distance of the inclusion $Q\cap \{\chi_M\}$ to the boundary of $Q$
			so that
			\[
			\|\nabla\hat w_\eps\|_{L^\infty(\Lambda;\R^{3\times 3})}\leq C \|\nabla\hat w\|_{L^\infty(\Lambda;\R^{3\times 3})}.
			\]
		\end{lemma}
		\begin{proof}
			
			Let $N_\varepsilon\subset Q_\varepsilon$ be an open set such that $\{x\in Q_\varepsilon:\, \chi_M(x/\varepsilon)=1\}\subset N_\varepsilon$. Moreover, we can assume that $N_\varepsilon$ has a smooth boundary. Let  $\varphi_\varepsilon\in C^\infty_0(Q_\varepsilon;[0,1])$   be such that $\varphi_\varepsilon(x)=1$ if $\chi_M(x/\varepsilon)=1$, $\varphi_\varepsilon(x)=0$ if $x\in Q_\varepsilon\setminus N_\varepsilon$ and $|\nabla\varphi_\varepsilon|\le K\varepsilon^{-1}$ for all $\varepsilon>0$ where $K>0$ does not depend on $\varepsilon$. Let $ \mu_\varepsilon=|N_\varepsilon|^{-1}\int_{N_\varepsilon}w\,{\rm d}x$ and $w_\varepsilon=\varphi_\varepsilon\mu_\varepsilon+(1-\varphi_\varepsilon)w$.
			Consequently, 
			$$\nabla w_\varepsilon=(1-\varphi_\varepsilon)\nabla w+( \mu_\varepsilon-w)\otimes\nabla\varphi_\varepsilon.$$
			The second term is bounded in $L^2(Q_\eps;\R^3)$ by the Poincar\'{e}-Wirtinger inequality and this yields a  bound on $\nabla w_\varepsilon$ as desired.

			To see that $w_\eps \rightharpoonup w$ weakly in $W^{1,2}$, first observe that as a bounded sequence, all its subsequences will have a weakly convergent subsequence. Therefore, it suffices to check that this weak limit always coincides with $w$. Actually, 
			$w_\eps \to w$ strongly in $L^2$:
			\begin{equation} \label{eq:hatwL2conv}
				\begin{aligned}
					&\int_{\Lambda} |\hat{w}_\eps-w|^2\,dx = \sum_{Q_\eps\in \cQ(\eps)} \int_{Q_\eps}|\hat{w}_\eps-w|^2\,dx  \\
					& \leq \sum_{Q_\eps\in \cQ(\eps)} C \eps^2 \int_{Q_\eps} |\nabla \hat{w}_\eps-\nabla w|^2 \\
					& \leq \sum_{Q_\eps\in \cQ(\eps)} 2C \eps^2 \int_{Q_\eps} (|\nabla \hat{w}_\eps|^2+|\nabla w|^2)
					\leq \eps^2 
					\tilde{C} \int_{\Lambda} |\nabla w|^2\to 0.
				\end{aligned}
			\end{equation}
			Here, we used the Poincar\'{e} inequality on each $\eps$-cell $Q_\eps$ (with the natural scaling of its constant), combined with the 
			properties \eqref{eq:hatu} of $\hat{w}_\eps$, in particular the bound
			for its gradient in terms of $w$.

		\end{proof}

		In addition, we will need a way to combine the divergence-free structure of the magnetic field with cut-off arguments. 
		The latter typically destroy differential constraints, but is is possible to project back with controlled error. 
		In fact, this can even be done in the more general context of so-called $\cA$-free fields, where $\cA=\div$ is one of many admissible examples \cite{FoMu99a}. In our special context, the following suffices:
		\begin{lemma}\label{lem:div-free-projection}
			Let $Q\subset \R^3$ be an open cube. Then exists a bounded linear projection
			\[
			\cP:L^2(Q;\R^3)\to L^2_{\Qper,\div}(\R^3;\R^3) :=
			\left\{{B}\in L^2_\loc(\R^3;\R^3) \,\left|\, 
			\begin{array}{l} 	
				{B}~~\text{is $Q$-periodic}\\
				\text{and}~~\div {B}=0~~\text{in $\R^3$}
			\end{array}\right.\right\}
			\]
			such that for all ${B}\in L^2_\Qper(\R^3;\R^3)$ 
			(i.e., ${B}\in L^2(Q;\R^3)$ interpreted as a periodic function on $\R^3$ by periodic extension),
			\[
			\|{B}-\cP {B}\|_{L^2(Q;\R^3)}\leq C \|\div{B}\|_{(W^{1,2}_\Qper(\R^3;\R^3))^*}
			\]
			with a constant $C>0$ independent of ${B}$.
		\end{lemma}
		\begin{proof} 
			One can apply \cite[Lemma 2.14]{FoMu99a} with $p=2$ and $\cA=\div$. 
			For the reader's convenience, we sketch a proof for our (easier and well-known) special case below.
			For each ${B}\in L^2_\Qper(\R^3;\R^3)$ define 
			\[
			\cP {B}:={B}-\nabla v,
			\]
			where $v=v({B})\in W^{1,2}_\Qper(\R^3)$ is given as the unique weak solution of
			\[
			\Delta v=\div {B}~~~\text{on $\R^3$}, ~~~\int_Q v = 0.
			\]
			It is not hard to check that $\cP$ now has all asserted properties. 
		\end{proof}
		
		\begin{lemma}\label{lem:div-free-proj-R3} 
			There exists a bounded linear projection
			\[
			\cP:L^2(\R^3;\R^3)\to L^2_{\div}(\R^3;\R^3):=
			\left\{B\in L^2(\R^3;\R^3) \,\left|\, 
			\begin{array}{l} 	
				\div {B}=0~~\text{in $\R^3$}
			\end{array}\right.\right\}
			\]
			such that for all ${B}\in L^2(\R^3;\R^3)$,
			\[
			\|{B}-\cP {B}\|_{L^2(\R^3;\R^3)}\leq C \|\div{B}\|_{(\mathring W^{1,2}(\R^3))^*}
			\]
			with a constant $C>0$ independent of ${B}$. Here,
			\[
			\mathring W^{1,2}(\R^3)):=\{[v]\mid v\in W^{1,2}_{loc}(\R^3),~ \|v\|_{\mathring W^{1,2}}^2:=\int_{\R^3} |\nabla v|^2<\infty\},
			\]
			and the equivalence class $[v]$ of $v$ consists of all functions $w$ with $\nabla v=\nabla w$ a.e.~in $\R^3$.
		\end{lemma}
		\begin{proof} 
			A much more general version of this can be found in \cite[Lemma 3.4]{Kroe11a}.
			For the reader's convenience, we give a simple proof of our special case here, similar to the one of Lemma~\ref{lem:div-free-projection}.
			We define $v$ (more precisely, its equivalence class) as the minimizer of $
			w\mapsto \int_{\R^3} \frac{1}{2}|\nabla w|^2+ (\div B) w\,dx$
			in $\mathring W^{1,2}(\R^3)$, and set $\cP {B}:={B}-\nabla v$ as before. 
			For the error estimate, notice that 
			$\div B$, interpreted as a linear functional in the dual space $(\mathring W^{1,2}(\R^3))^*$, is naturally defined as
			\[
			(\div B)[\varphi]:=-\int_{\R^3} B\cdot \nabla\varphi\,dx,\quad \varphi\in \mathring W^{1,2}(\R^3).
			\]
			Hence
			\[
			\begin{aligned}
				&\|{B}-\cP {B}\|_{L^2(\R^3;\R^3)} =\int_{R^3} |\nabla v|^2
				=\big|(\div B)[v]\big| \\
				&\qquad\leq \|\div B\|_{(\mathring W^{1,2}(\R^3))^*} \|v\|_{\mathring W^{1,2}(\R^3)}
				=\|\div B\|_{(\mathring W^{1,2}(\R^3))^*} \|\nabla v\|_{L^{2}(\R^3;\R^{3})}.
			\end{aligned}
			\]
		\end{proof}

		\subsection{Homogenization via Gamma-convergence}
		This subsection is dedicated to the proof of Theorem~\ref{thm:ConvLin}. We will loosely follow the strategy of M\"uller in \cite{Muller1987}. Compared to M\"uller's proof, we have to be significantly more careful when employing cut-off arguments, to preserve the constraint of rigidity on the inclusions. 
		Our arguments to handle the latter are somewhat similar to those of \cite{BraiGa95a} on the technical level, but our proof is self-contained with Lemma~\ref{lem:wlim-approx} as a crucial ingredient.
		In addition, the divergence-free magnetic fields defined on the whole space have to be handled with care.
		For this purpose, we exploit both Lemma~\ref{lem:div-free-projection} and Lemma~\ref{lem:div-free-proj-R3}.
		
		As in \cite{Muller1987}, we start our proof with a special case.
		To slightly simplify the notation whenever domain dependence temporarily plays a role, we will use the short-hand
		\begin{equation}\label{Edomains}	
			E_\eps(u,B;\Lambda):=
			\int_{(\Omega\cap\Lambda)\setminus \Omega_\eps}f_b(\nabla u,B)\,dx + 
			\int_{\Omega_\eps \cap\Lambda} f\Big(\frac{x}{\eps},\nabla u,B\Big)\,dx+\int_{\Lambda\setminus \Omega} f_{ext}(x,B)\,dx
		\end{equation}
		for any open $\Lambda\subset \R^3$. Be warned that with this definition, $E_\eps(u,B)=E_\eps(u,B;\R^3)\neq E_\eps(u,B;\Omega)$, as the contribution on $\R^3\setminus \Omega$ is missing in the latter.
		For the limit functional $E_{hom}$, $E_{hom}(u,B;\Lambda)$ is defined analogously.
		\begin{proposition}[The case of affine/constant weak limits] \label{prop:affine}
			Let $\Omega\subset \R^3$ be a bounded Lipschitz domain. We consider the functionals 
			$E_\eps$ and $E_{\rm hom}$ be given by \eqref{Eeps}, \eqref{Ehom} and \eqref{Edomains}, with integrands 
			$f$ and $f_{ext}$ satisfying \eqref{f0}--\eqref{f3}  and \eqref{fext0}--\eqref{fext2}. 
			Moreover, suppose that 
			\[
			\text{$u(x)=\lambda x$ in $\Omega$ and $B\in L^2_{\div}(\R^3;\R^3)$ with $B(x)=B_0$ in $\Omega$}, 
			\]
			with constant $\lambda\in \R^{3\times 3}$ and $B_0\in \R^3$. 
			Then for every bounded Lipschitz domain $\Lambda\subset \R^3$ as well as for $\Lambda=\R^3$, the following statements hold: 
			\begin{enumerate} 
				\item[(i)] ($\Gamma-\liminf$ inequality) For every sequence $(u^\eps,B^\eps)_{\eps>0}
				\subset W^{1,2}(\Omega;\R^3)\times L^2_{\div}(\R^3;\R^3)$ such that $(u^\eps,B^\eps)\rightharpoonup ( u, B)$  in $ W^{1,2}(\Omega;\R^3)\times L^2(\R^3;\R^3)$ for $\eps\to 0$ we have 
				\begin{align}\label{eq:LinLiminf}
					\liminf_{\eps\to 0} E_\eps( u^\eps,B^\eps ;\Lambda )\geq E_{\rm hom}( u, B  ; \Lambda );
				\end{align}
				\item[(ii)]  (almost recovery sequence) For every $\delta>0$, there exists a sequence $( u^{\delta,\eps}, B^{\delta,\eps}) \subset W^{1,2}(\Omega;\R^3)\times L^2_{\div}(\R^3;\R^3)$ 
				such that as $\eps\to 0$, $( u^{\delta,\eps}, B^{\delta,\eps})\rightharpoonup(u, B)$ weakly in 
				$ W^{1,2}(\Omega;\R^3)\times L^2(\R^3;\R^3)$ and
				\begin{align}\label{eq:LinLimsup} 
					\limsup_{\eps\to 0} E_\eps( u^{\delta,\eps},B^{\delta,\eps} ;\Lambda )\leq E_{\rm hom}(u, B ;\Lambda )+|\Lambda\cap\Omega|\delta. 
				\end{align} 
				In addition, the sequence can be chosen so that $|\nabla u^{\delta,\eps}|^2+|B^{\delta,\eps}|^2$ 
				is equi-integrable  for each fixed $\delta$. 
			\end{enumerate} 
		\end{proposition}
		\begin{proof}  For simplicity, we will only consider the case $\Lambda=\R^3$ below. 
			Restrictions of the functionals can be treated analogously. 
			
			\noindent{\bf(i) \emph{$\Gamma-\liminf$-inequality}:} \\
			\noindent{\bf Step i.1: Approximating $\Omega$ from inside with a cubical grid.}
			Let $u^{\varepsilon}\rightharpoonup \lambda x$ weakly in $W^{1,2}(\Omega;\R^3)$ and $B^{\varepsilon}\rightharpoonup B_0$ in $L^{2}_{\div}(\R^3;\R^3)$.
			Fix $\alpha>0$ and choose a cubical $\alpha$-grid subdividing $\R^3$ with open, pairwise disjoint cubes of side length $\alpha$, i.e., we can write
			\begin{equation*}
				\R^3=\bigcup_{\tilde{Q}\in \mathcal{G}_\alpha} \tilde{Q}+R,
			\end{equation*}
			where $R$ is a set of measure zero. 
			There is a constant $C=C(\Omega)>0$ (using that $\Omega$ is a Lipschitz domain)  such that
			for all $0<\alpha\leq 1$,
			\begin{equation*}
				\Big|\Omega\setminus \textstyle\bigcup_{\tilde{Q}\in \mathcal{G}_{\alpha,\Omega}} \tilde{Q}\Big|\leq C\alpha,~~~\text{where}~~
				\mathcal{G}_{\alpha,\Omega}:=\{\tilde{Q}\in \mathcal{G}_{\alpha}\mid \tilde{Q}\subset \Omega,~\operatorname{Dist(\tilde{Q};\partial\Omega)}\geq \alpha\}.
			\end{equation*} 
			Notice that long as $\eps\leq\alpha$, we also have 
			that 
			\[
				\bigcup_{\tilde{Q}\in \mathcal{G}_{\alpha,\Omega}} \tilde{Q}\subset \Omega_\eps.
			\] 
			By the lower bound on $f$ and $f_b$ in \eqref{f1} we can write
			\begin{equation}\label{eq:covering-by-tildeQs}
				\begin{aligned}
					& \int_{\Omega\setminus\Omega_\eps} f_b\bigg(\frac{x}{\eps},u^{\varepsilon}, B^{\varepsilon}\bigg) dx
					+\int_{\Omega_\eps} f\bigg(\frac{x}{\eps},u^{\varepsilon}, B^{\varepsilon}\bigg) dx \\
					&\geq \sum_{\tilde{Q}\in\mathcal{G}_{\alpha,\Omega}} \int_{\tilde{Q}} f\bigg(\frac{x}{\eps},u^{\varepsilon,i}, B^{\varepsilon}\bigg) dx 
					- \int_{\Omega\setminus \textstyle\bigcup_{\tilde{Q}\in \mathcal{G}_{\alpha,\Omega}} \tilde{Q}} C \,dx \\
					&\geq \sum_{\tilde{Q}\in\mathcal{G}_{\alpha,\Omega}} \int_{\tilde{Q}} f\bigg(\frac{x}{\eps},u^{\varepsilon,i}, B^{\varepsilon}\bigg) dx
					-C^2 \alpha.
				\end{aligned}
			\end{equation}
			Since $\alpha>0$ is arbitrary and the contribution $f_{ext}$ was assumed to be $L^2$-weakly lower semicontinuous on $\R^3\setminus\Omega$, it now suffices to show the $\liminf$ inequality on each cube $\tilde{Q}\in \mathcal{G}_{\alpha,\Omega}$.
			Let us fix one such $\tilde{Q}$.
			
			\noindent{\bf Step i.2: The $\Gamma-\liminf$ on a cube $\tilde{Q}\subset\subset \Omega_\eps$.}\\
			Choose the largest $k \in \mathbb{N}$ such that $k \varepsilon \leq \alpha - 2\varepsilon$ (recall that $\alpha$ is the side length of $\tilde{Q}$).   
			There exists $z_{0} \in \mathbb{Z}^3$ such that
			\begin{equation} \label{eq:defQeps}
				\tilde{Q}^{\varepsilon} := \varepsilon (z_{0} + k Q)\subset \tilde{Q}.
			\end{equation}
			Notice that $\tilde{Q}^{\eps}$ is essentially the largest possible multicell that is still fully contained in $\tilde{Q}$. 
			
			We now intend to modify the boundary values of $u^\eps$ on $\tilde{Q}^{\eps}$ in order to later exploit 
			the definition of the infimum in the multi-cell formula \eqref{eq:deffhom}.
			For this, we need a cut-off argument which plays a crucial role in the proof. 
			Ultimately, we will need that after cut-off of $u^\eps$, the new function function will have the boundary values $\lambda x$ on $\partial Q^\eps$.
			We cannot, however, use straightforward interpolation of $u^\eps$ and $\lambda x$, because this would break the rigidity constraint on $M$.
			Instead, we will first modify the affine function $\lambda x$.
			For this, we apply Lemma~\ref{lem:wlim-approx} 
			with $w(x):=\lambda x$ and a domain $\Lambda$ such that $\tilde{Q}\subset\subset \Lambda\subset \Omega$.
			This gives functions $\hat{\lambda}^\eps:=\hat{w}^\eps$ such that
			$\hat{\lambda}^\eps\rightharpoonup w=\lambda x$ in $W^{1,2}$, $\hat{\lambda}^\eps=\lambda x$ on $\partial Q^\eps$ 
			(actually even on the boundary of all single cells contained in $\tilde{Q}$) and $\nabla \hat{\lambda}^\eps=0$ on $\tilde{Q}\cap \eps M \big\}$ as long as $\eps$ is small enough. 
			(Here, the last assertion is the reason for introducing the slightly larger set $\Lambda$, because the lemma does not give 
			$\nabla \hat{\lambda}^\eps=0$ fully up to the boundary of $\Lambda$.)

			Let us consider a $Q_0$ compactly embedded and open in $\tilde{Q}$ and define
			\begin{equation}
				\eta := \frac{1}{2} d(Q_0, \tilde{Q}) \ \ \mbox{and} \ \ Q_i= \bigg\{ x \in Q \bigg| d(x, Q_0) \leq \frac{i}{\nu} \eta  \bigg\}, \ \mbox{ where $i=1, \dots, \nu$}
			\end{equation} 
			Let us remark that $Q_0 \subset \subset Q_1 \subset \subset \dots \subset \subset Q_\nu \subset \subset \tilde{Q}$. We define scalar functions  $\varphi_i^\eps,\psi_i \in C^{\infty}_{c}(\tilde{Q})$ such that $0 \leq \varphi_i^\eps \leq 1$, 
			$0\leq \psi_i \leq 1$,
			$\varphi_i^\eps=\psi_i=1$ in $Q_{i-1}$, $\varphi_i^\eps=\psi_i=0$ in $\tilde{Q} \setminus Q_i$, 
			\begin{equation}\label{eq:phii-const-on-inc}
				\nabla \varphi_i^\eps=0~~\text{on}~~\Big\{x\in \tilde{Q}\,\Big|\,\chi_M\big(\tfrac{x}{\eps}\big)=1\Big\}~~~
				\text{and on}~~\tilde{Q},~~~|\nabla \varphi_i^\eps| \leq \frac{(\nu+1)}{\sigma \eta},~~~|\nabla \psi_i| \leq \frac{(\nu+1)}{\eta},
			\end{equation} 
			where $\sigma$ is the distance of the inclusions to the boundary of a reference cell, i.e., $\sigma:=\operatorname{dist}(Q\cap \{\chi_M=1\};\partial Q)$. We define
			\begin{equation*}
				u^{\varepsilon, i}(x) := \hat{\lambda}_{\varepsilon}(x) + \varphi_i^\eps(x) (u^{\varepsilon}(x) - \hat{\lambda}_{\varepsilon} (x))
			\end{equation*}
			and
			\begin{equation*}
				B^{\varepsilon, i} := \psi_i(B_{\varepsilon} - B_{0})+B_0.
			\end{equation*}
			As defined, $B^{\varepsilon, i}$ is not divergence free, but this we can achieve using the projection of
			Lemma \ref{lem:div-free-projection} on the cube $\tilde{Q}$, by setting $\tilde{B}^{\varepsilon, i} := \mathcal{P}(B^{\varepsilon, i}-B_0)+B_0$.
			The error introduced by projecting satisfies
			\begin{equation}\label{eq:Bprojerror}
				\| \tilde{B}^{\varepsilon, i} - B^{\varepsilon, i} \|_{L^2} \leq C \| div B^{\varepsilon, i} \|_{(W^{1,2}_{\Qper})^{*}}
				\underset{\eps\to 0}{\longrightarrow}0. 
			\end{equation}	
			Here, the latter holds because 
			\begin{equation*}
				div B^{\varepsilon, i} = (\nabla \psi_i) (B_{\varepsilon} - B_{0}) + \psi_i \  div (B_{\varepsilon} - B_{0}) + \div B_0 = (\nabla \psi_i) (B_{\varepsilon} - B_{0})
			\end{equation*}
			which  for fixed $\psi_i$ converges weakly to zero  in $L^2$ and thus strongly in $(W^{1,2}_{\tQper})^{*}$.
			
			We now have to link $E_\eps(u^{\varepsilon, i},\tilde{B}^{\varepsilon, i};\tilde{Q})$ to $E_\eps(u^{\varepsilon},B^{\varepsilon};\tilde{Q})$ 
			(recall \eqref{Edomains} for the definition of these energies restricted to cubes).
			Decompose the former as follows:
			\begin{align} \label{eq-aux0}
				\begin{aligned}
					\int_{\tilde{Q}}  f\bigg(\frac{x}{\eps},\nabla u^{\varepsilon, i}, \tilde{B}^{\varepsilon, i}\bigg) dx  
					& =  \int_{Q_{i-1}} f\bigg(\frac{x}{\eps},\nabla u^{\varepsilon}, \tilde{B}^{\varepsilon, i}\bigg) dx   
					+  \int_{Q_{i} \setminus Q_{i-1}} f \bigg(\frac{x}{\eps},\nabla u^{\varepsilon, i}, \tilde{B}^{\varepsilon, i}\bigg) dx  \\ 
					& + \int_{\tilde{Q} \setminus Q_{i}}  f\bigg(\frac{x}{\eps},\nabla \hat\lambda^{\varepsilon}, \tilde{B}^{\varepsilon, i}\bigg) dx
				\end{aligned}
			\end{align}
			By the uniform estimate of $\nabla\hat\lambda^{\eps}=\nabla\hat w^{\eps}$ in terms of $\nabla w=\lambda$ of Lemma \ref{lem:wlim-approx},
			we know that 
			\[
			\|\nabla\hat\lambda^{\eps} \|_{L^\infty}\leq C |{\lambda}|. 
			\]
			From the upper bound in \eqref{f1}, also using that $B^{\varepsilon, i}=B_0$ on $\tilde{Q} \setminus Q_{i}$, 
			we thus immediately obtain that
			\begin{equation} \label{eq-aux1}
				\begin{aligned}
					& \int_{\tilde{Q} \setminus Q_{i}} f\bigg(\frac{x}{\eps},\nabla \hat\lambda^{\varepsilon}, \tilde{B}^{\varepsilon, i}\bigg) dx \\
					&\leq C(1+|\lambda|^2) |\tilde{Q} \setminus Q_{i}|
					+ C\int_{\tilde{Q} \setminus Q_{i}} |\tilde{B}^{\varepsilon, i}|^2\,dx  \\
					&\leq C(1+|\lambda|^2) |\tilde{Q} \setminus Q_{i}|
					+ 2C|\tilde{Q} \setminus Q_{i}||B_0|^2+2C\int_{\tilde{Q} \setminus Q_{i}}|\tilde{B}^{\varepsilon, i}-B^{\varepsilon, i}|^2\,dx
				\end{aligned}
			\end{equation}
			In addition, observing that $\nabla u^{\varepsilon, i}= \nabla \hat{\lambda}_{\varepsilon}(x) + (\nabla \varphi_i^\eps(x)) (u^{\varepsilon} - \hat{\lambda}_{\varepsilon}(x)) + \varphi_i^\eps (\nabla u^{\varepsilon}- \nabla \hat{\lambda}_{\varepsilon})$, we also get that

		\begin{equation} 
			\begin{aligned}
				&\int_{Q_{i} \setminus Q_{i-1}} f\bigg(\frac{x}{\eps},\nabla u^{\varepsilon, i}, \tilde{B}^{\varepsilon, i}\bigg)\,dx
				\\
				& \leq \begin{aligned}[t] 
					&\int_{Q_{i} \setminus Q_{i-1}} C \Big(1+|\lambda| +  (u^{\varepsilon} - \hat\lambda_{\varepsilon}(x))\otimes \nabla \varphi_i^\eps(x) 
					+ \varphi_i^\eps (\nabla u^{\varepsilon}- \nabla \hat{\lambda}_{\varepsilon}(x))|^2\Big)\,dx \\ 
					&+2C \int_{Q_{i} \setminus Q_{i-1}} |\tilde{B}^{\varepsilon, i}-B^{\varepsilon, i}|^2\,dx
					+2C \int_{Q_{i} \setminus Q_{i-1}} |B^{\varepsilon, i}|^2\,dx
				\end{aligned}\\
				& \leq \begin{aligned}[t] 
					&\int_{Q_{i} \setminus Q_{i-1}} C \Big(1+|\lambda| + (\nabla \varphi_i^\eps(x)) (u^{\varepsilon} - \hat\lambda_{\varepsilon}(x)) 
					+ \varphi_i^\eps (\nabla u^{\varepsilon}- \nabla \hat{\lambda}_{\varepsilon}(x))|^2\Big)\,dx \\ 
					&+2C \int_{Q_{i} \setminus Q_{i-1}} |\tilde{B}^{\varepsilon, i}-B^{\varepsilon, i}|^2\,dx
					+4C |Q_{i} \setminus Q_{i-1}| |B_0|^2
					+4C \int_{Q_{i} \setminus Q_{i-1}} |\psi_i (B^\eps-B_0)|^2\,dx
				\end{aligned}
			\end{aligned} \label{eq-aux2}
		\end{equation}
		Finally, recalling that $B^{\varepsilon}=B^{\varepsilon,i}$ on $Q_{i-1}$, we can 
		exploit \eqref{f3} to estimate the ``inner'' term in \eqref{eq-aux0} as
		\begin{equation} \label{eq-aux3}
			\begin{aligned}
				\int_{Q_{i-1}} f\bigg(\frac{x}{\eps},\nabla u^{\varepsilon, i}, \tilde{B}^{\varepsilon, i}\bigg)\,dx 
				& \leq 
				\begin{aligned}[t] & \int_{Q_{i-1}} f\bigg(\frac{x}{\eps},\nabla u^{\varepsilon, i}, B^{\varepsilon}\bigg)\,dx
					+ \hat{C}(\nu) \|\tilde{B}^{\varepsilon, i}-B^{\varepsilon, i}\big\|_{L^2(Q_{i-1};\R^3)},
				\end{aligned}
			\end{aligned}
		\end{equation}
		where 
		\begin{equation*} 
			\begin{aligned}
				\hat{C}(\nu):=\max_{i} \sup_{0<\eps<1} C \Big(1+\big\|\nabla u^{\varepsilon, i}\big\|_{L^2(Q_{i-1};\R^3)}
				+\big\|B^{\varepsilon, i}\big\|_{L^2(Q_{i-1};\R^3)}
				+ \big\|\tilde{B}^{\varepsilon, i}\big\|_{L^2(Q_{i-1};\R^3)}\Big)<\infty
			\end{aligned}
		\end{equation*}
		since $\nabla \varphi_i^\eps(x)$, $B^{\varepsilon, i}$ and $\tilde{B}^{\varepsilon, i}$ are bounded for fixed $\nu$.
		Since $\tilde{B}^{\eps, i}-B^{\eps, i}\to 0$ as $\eps\to 0$ in $L^2$ by \eqref{eq:Bprojerror},
		we can plug \eqref{eq-aux1}--\eqref{eq-aux3} into \eqref{eq-aux0}. We thus infer that
		for any subsequence of $\eps$ (not relabeled) such that the following limit exists, 
		\begin{align*}
			& \lim_{\varepsilon \rightarrow 0}\int_{\tilde{Q}} f\bigg(\frac{x}{\eps},\nabla u^{\varepsilon, i}, \tilde{B}^{\varepsilon, i}\bigg) dx  \\ 
			& \begin{aligned}[t]
				& \leq \liminf_{\varepsilon \rightarrow 0}\int_{Q_{i-1}} f\bigg(\frac{x}{\eps},\nabla u^{\varepsilon}, B^{\varepsilon}\bigg) dx  
				+\tilde{C} \int_{Q_{i} \setminus Q_{i-1}} \big(|\nabla u^{\varepsilon}- \nabla \hat{\lambda}_{\varepsilon}|^2+|B^\eps-B_0|^2\big)\,dx
				+ \tilde{C}| \tilde{Q} \setminus Q_{i}|
				\\
				& \leq \liminf_{\varepsilon \rightarrow 0}\int_{\tilde{Q}} f\bigg(\frac{x}{\eps},\nabla u^{\varepsilon}, B^{\varepsilon}\bigg) dx 
				+\tilde{C} \int_{Q_{i} \setminus Q_{i-1}} \big(|\nabla u^{\varepsilon}- \nabla \hat{\lambda}_{\varepsilon}|^2+|B^\eps-B_0|^2\big)\,dx
				+ 2\tilde{C}| \tilde{Q} \setminus Q_{0}|.
			\end{aligned}
		\end{align*}
		Here, $\tilde{C}:=\max\{C(1+|\lambda|^2),6C |B_0|^2\}$
		and for the last inequality we also used the lower bound in \eqref{f1}.
		Summing over $i$ and dividing by $\nu$, we obtain that
		\begin{equation*}
			\liminf_{\eps \rightarrow 0} \frac{1}{\nu}\sum_i^\nu E^{\eps}(u^{\varepsilon, i}, B^{\varepsilon,i};\tilde{Q}) \leq 
			\begin{aligned}[t]
				&\liminf_{\varepsilon \rightarrow 0} 
				\int_{\tilde{Q}} f\bigg(\frac{x}{\eps},\nabla u^{\varepsilon}, B^{\varepsilon}\bigg)dx + 2\tilde{C} | \tilde{Q} \setminus Q_0| \\
				& + \limsup_{\varepsilon \rightarrow 0}  \frac{\tilde{C}}{\nu} \int_{\tilde{Q} } 
				\big(|\nabla u^{\varepsilon}- \nabla \hat{\lambda}_{\varepsilon}(x)|^2+|B^\eps-B_0|^2 \big)\,dx .
			\end{aligned}
		\end{equation*}
		Since $\| \nabla u^{\varepsilon} - \nabla \hat{\lambda}_{\varepsilon}(x) \|_{L^2}+\| B^{\varepsilon} - B_0 \|_{L^2}$
		is bounded, we conclude that
		\begin{equation} \label{eq:liminf-average}
			\lim_{|\tilde{Q}\setminus Q_0|\to 0}\lim_{\nu\to\infty} \liminf_{\eps \rightarrow 0} 
			\frac{1}{\nu}\sum_{i=1}^\nu  E^{\eps}(u^{\varepsilon, i}, B^{\varepsilon,i};\tilde{Q}) 
			\leq \liminf_{\varepsilon \rightarrow 0} E(u^\eps,B^\eps;\tilde{Q}).
		\end{equation}
		It now suffice to show that for each $i$,
		\begin{equation} \label{eq:liminf-on-tildeQ}
			E(\lambda x, B_0;\tilde{Q}) \leq \liminf_{\varepsilon \rightarrow 0}E^{\varepsilon}(u^{\varepsilon, i}, B^{\varepsilon,i};\tilde{Q}).
		\end{equation}
		For the proof of \eqref{eq:liminf-on-tildeQ}, by the lower bound on $f$ in \eqref{f1} we can write
		\begin{equation} \label{eq-aux10}
			\int_{\tilde{Q}} f\bigg(\frac{x}{\eps},\nabla u^{\varepsilon,i}, B^{\varepsilon,i}\bigg) dx 
			\geq \int_{\tilde{Q}^{\eps}} f\bigg(\frac{x}{\eps},\nabla u^{\varepsilon,i}, B^{\varepsilon,i}\bigg) dx - \int_{\tilde{Q}^{\eps}\setminus\tilde{Q}} C \,dx
		\end{equation}
		The second integral converges to zero as $\varepsilon \rightarrow 0$. Using the notation of \eqref{eq:defQeps},
		the substitution $x=\varepsilon(z_0 + y)$ and setting 
		$z^{\varepsilon} = u^{\varepsilon,i}-\lambda x \in W^{1,p}_{0}(\tilde{Q}^{\eps})$ and $\beta:=B^{\varepsilon,i}-B_0$, we see that
		\begin{align}\label{eq-aux11}
			\begin{aligned}
				\int_{\tilde{Q}^{\eps}} f\bigg(\frac{x}{\eps},\nabla u^{\varepsilon,i}, B^{\varepsilon,i}\bigg) dx 
				& = \varepsilon^3 \int_{kQ} f\bigg(y,\lambda + \nabla z^{\varepsilon} (\varepsilon(z_0 + y)) , B_{0} 
				+ \beta \big( \varepsilon(z_0 + y) \big)\bigg) \,dy \\
				& \geq (\varepsilon k)^3 f_{hom} (\lambda, B)
			\end{aligned}
		\end{align} 
		by the definition of $f_{hom}$ in \eqref{eq:deffhom}.
		Since $(\eps k)^3=(\eps k(\eps,\alpha))^3\to \alpha^3=|\tilde Q|$ as $\eps\to 0$, we can combine \eqref{eq-aux10} and \eqref{eq-aux11} to obtain \eqref{eq:liminf-on-tildeQ}. 
		In view of \eqref{eq:liminf-average}, the assertion follows from \eqref{eq:covering-by-tildeQs} in the limit as $\alpha\to 0$.		
		{\bf (ii) Almost recovery sequence}~\\
		{\bf Step ii.1: $\delta$-almost recovery sequence with errors from a transition layer.}
		Fix $\delta>0$ and choose $k^\delta\in \N$, $\phi^\delta\in W^{1,2}_{\Qper}(\R^3;\R^3)$, $\beta_0^\delta\in L^2_{\Qper,\div}(\R^3;\mathbb R^3)$ (almost) optimal with error at most $\delta$ for the infimum in \eqref{eq:deffhom}, the definition of $f_{hom}$.
		Define
		\begin{equation} \label{eq:urec}
			u^{\eps,\delta}  := \lambda x + k^\delta\varepsilon \phi^\delta\big( \frac{x}{k^\delta\varepsilon} \big),\quad x\in\Omega,
		\end{equation}
		and, with a parameter $0<r\leq 1$ that will govern an transition layer outside of $\partial\Omega$ of width $r$,
		\begin{equation} \label{eq:Brec-1}
			\tilde{B}^{\eps,\delta}_r(x):=
			(1-\psi_r(x)) B(x) + \psi_r(x) \Big(B(x) + \beta_{0}^\delta \big( \frac{x}{k^\delta\varepsilon} \big)\Big),\quad x\in \R^3.
		\end{equation}
		Here, for each $r$, $\psi_r\in C^\infty_c(\R^3;[0,1])$ denotes a ``cut-off''-function such that 
		\[
		\text{$\psi_r=1$ on $\Omega$},\quad \operatorname{supp}(\psi_r)\subset \Omega^{(r)}:=\{x\in\R^3\mid \operatorname{dist}(x;\Omega)<1\}.
		\]
		Notice that as defined, $\tilde{B}^{\eps,\delta}_r=B_0+\beta_{0}^\delta \big( \frac{x}{k^{\delta} \varepsilon} \big)$ on $\Omega$ and
		$\tilde{B}^{\eps,\delta}_r\rightharpoonup B$ for each $r>0$. 
		Moreover, since $\beta_{0}^\delta$ is a fixed function in $L^2_\Qper$, it is not difficult to see that the family $\big(\big|\beta_{0}^\delta \big( \frac{\cdot}{k^{\delta}\varepsilon} \big)\big|^2\big)_{0<\eps\leq 1}$ 
		is locally equi-integrable on $\R^3$,  and therefore also $(\big|\tilde{B}^{\eps,\delta}_r\big|^2)_\eps$ {is locally equi-integrable}, uniformly in $r$ for each fixed $\delta>0$. Similarly, $(\big|\nabla u^{\eps,\delta}\big|^2)_\eps$ is equi-integrable on $\Omega$ for fixed $\delta$. 
		As a consequence, 
		we have that
		\begin{equation} \label{eq:Brec-strongcovnergence}
			\sup_{0<\eps\leq 1}\int_{\R^3\setminus \Omega} |\tilde{B}^{\eps,\delta}_{r}-B|^2\,dx \leq \sup_{0<\eps\leq 1}\int_{\Omega^{(r)}\setminus \Omega} 
			\Big|\beta_{0}^\delta \big( \frac{x}{k^{\delta} \varepsilon} \big)\Big|^2\,dx\underset{r\to 0}{\longrightarrow}0.
		\end{equation} 
		However, $\tilde{B}^{\eps,\delta}_r\in L^2(\R^3;\R^3)$ is in general not divergence-free due to the transition near $\partial\Omega$ created by $\psi_r$ in its definition. Ultimately, we therefore have to further modify it using the projection $\mathcal{P}$ of Lemma~\ref{lem:div-free-proj-R3}.
		
		{\bf Step ii.2: Error estimate for the projection to divergence-free fields and choice of a diagonal sequence with vanishing transition layer.}
		As stated in Lemma~\ref{lem:div-free-proj-R3}, the projection error satisfies
		\begin{equation} \label{eq-aux20}
			\|\cP\tilde{B}^{\eps,\delta}_r-\tilde{B}^{\eps,\delta}_r\|_{L^2(\R^3;\R^3)}\leq C \| \div \tilde{B}^{\eps,\delta}_r \|_{(\mathring W^{1,2}(\R^3))^*}.
		\end{equation}
		To control it, it suffices to show that $\div \tilde{B}^{\eps,\delta}_r\to 0$ as $\eps\to 0$ strongly in $(\mathring W^{1,2}(\R^3))^*$. As to the latter, observe that for all $v\in \mathring W^{1,2}(\R^3)$,
		\begin{equation} \label{eq-aux21}
			\int_{R^3} \tilde{B}^{\eps,\delta}_r\cdot \nabla v\,dx
			=\int_{R^3} B\cdot \nabla v\,dx+\int_{R^3} \psi_r\beta_{0}^\delta \big( \frac{x}{k^{\delta} \varepsilon} \big) \cdot \nabla v\,dx=
			\int_{\Omega^{(1)}} \psi_r\beta_{0}^\delta \big( \frac{x}{k^{\delta} \varepsilon} \big) \cdot \nabla v\,dx,
		\end{equation}
		since $\div B=0$ and $\psi_r$ is supported in $\Omega^{(1)}$. For each $\eps>0$, choose a sequence $(v_{\eps,k})_k\subset \mathring W^{1,2}(\R^3)$
		with $\|v_{\eps,k}\|_{\mathring W^{1,2}(\R^3)}\leq 1$ for all $\eps,k$ such that
		\begin{equation*}
			\| \div \tilde{B}^{\eps,\delta}_r \|_{(\mathring W^{1,2}(\R^3))^*}
			=\sup_{\|v\|_{\mathring W^{1,2}(\R^3)}\leq 1} \Big|\int_{\R^3} \tilde{B}^{\eps,\delta}_r\cdot \nabla v\,dx\Big|
			=\lim_{k\to\infty} \int_{\R^3} \tilde{B}^{\eps,\delta}_r\cdot \nabla v_{\eps,k}\,dx
		\end{equation*}
		With $\tilde{v}_{\eps,k}:=v_{\eps,k}-\frac{1}{|\Omega^{(1)}|}\int_{\Omega^{(1)}}v_{\eps,k}\,dx \in \mathring W^{1,2}(\Omega^{(1)})$
		(in fact, we even have that $\tilde{v}_{\eps,k}\in W^{1,2}(\Omega^{(1)})$ by Poincar\'e's inequality) we further get that
		\begin{equation*}
			\| \div \tilde{B}^{\eps,\delta}_r \|_{(\mathring W^{1,2}(\R^3))^*}
			=\lim_{k\to\infty} \int_{\Omega^{(1)}}\psi\beta_{0}^\delta \big( \frac{x}{k^{\delta} \varepsilon} \big) \cdot \nabla \tilde{v}_{\eps,k}\,dx
		\end{equation*}
		Since $\div \beta_0=0$ and $\psi$ is compactly supported in $\Omega^{(1)}$, we can integrate by parts and obtain that
		\begin{equation*}
			\| \div \tilde{B}^{\eps,\delta}_r \|_{(\mathring W^{1,2}(\R^3))^*}
			=\lim_{k\to\infty} -\int_{\Omega^{(1)}} (\nabla \psi)\cdot \beta_{0}^\delta \big( \frac{x}{k^{\delta} \varepsilon} \big) \tilde{v}_{\eps,k}\,dx
			=:a_{\eps,k}.
		\end{equation*}
		By the compact embedding of $W^{1,2}(\Omega^{(1)})$ into $L^2(\Omega^{(1)})$,
		the set $\{\tilde{v}_{\eps,k}\mid \eps>0,~k\in \N\}$ is compact in $L^2(\Omega^{(1)})$.
		Choose a sequence $\eps_j\to 0$ such that 
		$\lim_{j\to\infty}\| \div \tilde{B}^{\eps_j} \|_{(\mathring W^{1,2}(\R^3))^*}
		=\limsup_{\eps\to 0}\| \div \tilde{B}^{\eps} \|_{(\mathring W^{1,2}(\R^3))^*}$
		and $k(j)\in\N$ such that $a_{\eps_j,k(j)}\leq \lim_k a_{\eps_j,k}+\frac{1}{j}$.
		Then
		\begin{equation*} 
			\| \div \tilde{B}^{\eps,\delta}_r \|_{(\mathring W^{1,2}(\R^3))^*}
			=\lim_{j\to\infty} a_{\eps_j,k(j)}
			=\lim_{j\to\infty} -\int_{\Omega^{(1)}} (\nabla \psi)\cdot \beta_{0}^\delta \big( \frac{x}{k^{\delta} \varepsilon} \big) \tilde{v}_{\eps_j,k(j)}\,dx
		\end{equation*}
		Passing to another subsequence if necessary, $\tilde{v}_{\eps_j,k(j)}\to v$ strongly in $L^2(\Omega^{(1)})$, with a suitable limit $v$.
		In addition, we know that $\beta_{0}^\delta \big( \frac{x}{k^{\delta}\varepsilon} \big)\rightharpoonup \int_Q \beta_{0}^\delta(z)\,dz=0$ 
		by the Riemann-Lebesgue lemma. Altogether,
		\begin{equation} \label{eq-aux25}
			\limsup_{\eps\to 0}\| \div \tilde{B}^{\eps,\delta}_r \|_{(\mathring W^{1,2}(\R^3))^*}
			\leq\lim_{j\to\infty} -\int_{\Omega^{(1)}} (\nabla \psi)\cdot \beta_{0}^\delta \big( \frac{x}{k^{\delta}\varepsilon} \big) \tilde{v}_{\eps_j,k(j)}\,dx=0.
		\end{equation}
		We now define
		\begin{equation} \label{eq:defBrecaffine}
			B^{\eps,\delta}:=B^{\eps,\delta}_{r(\eps)}=\cP \tilde{B}^{\eps,\delta}_{r(\eps)} \in L^2_{\div}(\R^3;\R^3),
		\end{equation}
		where $r(\eps)\to 0$ is chosen slow enough so that 
		\begin{equation}\label{eq:conv-divB}
			\lim_{\eps\to 0}\| \div \tilde{B}^{\eps,\delta}_{r(\eps)} \|_{(\mathring W^{1,2}(\R^3))^*}\to 0	
		\end{equation}
		which is possible due to \eqref{eq-aux25}.
		Due to \eqref{eq-aux20}, this implies that
		\begin{equation} \label{eq-aux26}
			B^{\eps,\delta}-\tilde{B}^{\eps,\delta}_{r(\eps)} \underset{\eps\to 0}{\longrightarrow} 0\quad\text{strongly in }L^2(\R^3;\R^3),
		\end{equation}
		for every fixed $\delta>0$.  
		In addition, since $\phi^\delta$ and $\beta_0^\delta$ are periodic functions with zero average, we have that 
		$(u^{\eps,\delta}, \tilde{B}^{\eps,\delta}_{r(\eps)}) \rightharpoonup(\lambda x, B_{0})$ as $\varepsilon \rightarrow 0$, weakly in 
		$W^{1,2}(\Omega;\R^3)\times L^2(\Omega;\R^3)$. 
		In view of \eqref{eq:Brec-strongcovnergence}, \eqref{eq-aux26} yields that 
		$B^{\eps,\delta}\to B$ strongly in $L^2(\R^3\setminus \Omega;\R^3)$, and, since $B=B_0$ on $\Omega$, also
		$B^{\eps,\delta}\rightharpoonup B$ weakly $L^2(\R^3;\R^3)$. 
		
		\noindent{\bf Step ii.3: Estimating the energy along $(u^{\eps,\delta},B^{\eps,\delta})$.}
		Because of the periodicity of $f$, the function
		\begin{equation*}
			x\mapsto f \bigg( \frac{x}{\varepsilon}, \lambda + \nabla \phi^{\delta} \big( \frac{x}{k^\delta\eps} \big)  , B_{0} + \beta_0^\delta \big( \frac{x}{k^\delta\eps} \big)\bigg)
		\end{equation*}
		is periodic on cubes of side $k^\delta\eps$, hence it converges weakly to its average. By a change of variables $z=x/(\eps k^\delta)$, which maps the cube $\eps k^\delta Q$ to the cube $Q$, and by invoking the definition of $\phi^\delta$ and $\beta_0^\delta$, we see that this average is estimated as follows:
		\begin{equation*}
			\begin{aligned}
				\frac 1 {(\eps k^\delta)^3}\int_{\eps k^\delta Q}	f \bigg( \frac{x}{\varepsilon}, \lambda + \nabla \phi^{\delta} \big( \frac{x}{k^\delta\eps} \big)  , B_{0} + \beta_0^\delta \big( \frac{x}{k^\delta\eps} \big)\bigg)dx&=
				\int_{Q}	f \big( kz, \lambda + \nabla \phi^{\delta} (z)  , B_{0} + \beta_0^\delta(z)\big)dz\\
				&\le f_{hom}(\lambda,B_0)+\delta.
			\end{aligned}
		\end{equation*}
		Then, since $\tilde B^{\eps,\delta}_{r(\eps)}$ coincides with $B_0+\beta_0^\delta(x/(k^\delta\eps))$ in $\Omega$, the growth conditions \eqref{f1} on $f_b$ and \eqref{f1}, \eqref{f2} on $f$, the $2$-equi-integrability
		of $\nabla u^{\eps,\delta}$ and $\tilde B^{\eps,\delta}_{r(\eps)}$ on $\Omega$ and the	
		above-mentioned convergence yield that
		\begin{equation}\label{eq:limsup}
			\begin{aligned}
				& \limsup_{\eps\to 0} E_{\varepsilon}(u^{\eps,\delta}, \tilde{B}^{\eps,\delta}_{r(\eps)} ;\Omega ) \\
				& \leq 
				\limsup_{\eps\to 0}\int_{\Omega\setminus\Omega_\eps} 	f_b \Big( \frac{x}{\varepsilon}, \nabla u^{\eps,\delta}, \tilde B^{\eps,\delta}_{r(\eps)}\Big)\,dx
				+\limsup_{\eps\to 0}\int_{\Omega_\eps} 	f \Big( \frac{x}{\varepsilon}, \nabla u^{\eps,\delta}, \tilde B^{\eps,\delta}_{r(\eps)}\Big)\,dx 
				\\ 
				 & \leq \limsup_{\eps\to 0} C \int_{\Omega\setminus\Omega_\eps} (1+|\nabla u^{\eps,\delta}|^2+|\tilde B^{\eps,\delta}_{r(\eps)}|^2)\,dx
				+\limsup_{\eps\to 0}\int_{\Omega} 	f \Big( \frac{x}{\varepsilon}, \nabla u^{\eps,\delta}, \tilde B^{\eps,\delta}_{r(\eps)}\Big)\,dx 
				\\ 
				&=\limsup_{\eps\to 0}\int_{\Omega} 	f \Big( \frac{x}{\varepsilon}, \nabla u^{\eps,\delta}, \tilde B^{\eps,\delta}_{r(\eps)}\Big)\,dx 
				\\
				&=\limsup_{\eps\to 0}\int_\Omega 	f \bigg( \frac{x}{\varepsilon}, \lambda + \nabla \phi^{\delta} \big( \frac{x}{k^\delta\eps} \big)  , B_{0} + \beta_0^\delta \big( \frac{x}{k^\delta\eps} \big)\bigg)\,dx\\
				&\le |\Omega|f_{hom}(\lambda,B_0)+|\Omega|\delta = E_{hom}(\lambda,B_0;\Omega )+|\Omega|\delta.
			\end{aligned}
		\end{equation}
		Next, we estimate the error in plugging the projection of the magnetic fields:
		\begin{equation}\label{eq:lim}
			\begin{aligned}
				&|E_\eps(u^{\eps,\delta},B^{\eps,\delta};\Omega)-E_\eps(u^{\eps,\delta},\tilde B^{\eps,\delta}_{r(\eps)};\Omega)|\\
				&\qquad \le |\Omega|C\Big(\|\nabla u^{\eps,\delta}\|_{L^2(\Omega;\R^{3\times 3})}+\|\tilde B^{\eps,\delta}_{r(\eps)}\|_{L^2(\Omega;\R^3)}+\|B^{\eps,\delta}\|_{L^2(\Omega;\R^3)}\Big)\|B^{\eps,\delta}-\tilde B^{\eps,\delta}_{r(\eps)}\|_{L^2(\Omega;\R^3)}
				\\
				&\qquad \le C\|\div B^{\varepsilon,\delta}_{r(\eps)}\|_{ (W^{1,2}(\R^3))^* }\to 0\qquad\text{as }\varepsilon\to 0.
			\end{aligned}
		\end{equation}
		Here, the first inequality follows from \eqref{f3}. To obtain the second inequality, we have used the fact that, by  construction, $\nabla u^{\eps,\delta}$ and $\tilde B^{\eps,\delta}_{r(\eps)}$ are bounded, respectively, in $L^2(\Omega;\R^{3\times 3})$ and in $L^2(\Omega;\R^3)$, and that,  furthermore, the projection operator is also bounded from $L^2$ to $L^2$, so that also $B^{\eps,\delta}$ is bounded in $L^2(\Omega;\R^3)$. To obtain the convergence in the third line of \eqref{eq:lim} we have used \eqref{eq-aux20} and \eqref{eq:conv-divB}.
		
		Combined, \eqref{eq:lim} and \eqref{eq:limsup} imply that
		\begin{equation}\label{rec-aff-Omega0}
			\limsup_{\eps\to 0} E_{\varepsilon}(u^{\eps,\delta}, B^{\eps,\delta};\Omega)
			\leq E_{hom} (\lambda, B_0;\Omega)+|\Omega|\delta=E_{hom} (\lambda, B;\Omega)+|\Omega|\delta.
		\end{equation}
		Moreover, 
		\begin{equation}\label{rec-aff-Omega}
			\begin{aligned}
				\limsup_{\eps\to 0} E_{\varepsilon}(u^{\eps,\delta}, B^{\eps,\delta};\R^3\setminus\Omega)
				&=\limsup_{\eps\to 0} \int_{\R^3\setminus\Omega} f_{ext}(x,B^{\eps,\delta})\,dx \\
				&=\int_{\R^3\setminus\Omega} f_{ext}(x,B)\,dx=E_{hom} (\lambda, B;\R^3\setminus\Omega),
			\end{aligned}
		\end{equation}
		as a consequence of \eqref{eq-aux26} and the $L^2$-continuity on $\R^3\setminus\Omega$ assumed in \eqref{fext1}.
		Altogether, we have shown that $(u_\eps,B_\eps):=(u^{\eps,\delta(\eps)},B^{\eps,\delta(\eps)})$ 
		is a $\delta$-almost recovery sequence as asserted.
		Here, also notice that for fixed $\delta$, $|B^{\eps,\delta}|^2$ inherits the local equi-integrability of $|\tilde{B}_r^{\eps,\delta}|^2$ (which is uniform in $r$) due to \eqref{eq-aux26}.
		
	\end{proof}
	 
	The previously obtained characterization of the homogenization functional for affine/constant limits allows us to prove that $f_{hom}$ is $2$-Lipsichitz continuous, i.e., locally Lipschitz.
	\begin{proposition}[properties of $f_{hom}$] \label{prop:fhom}
		Suppose that $f$ satisfies \eqref{f0}--\eqref{f3}.
		Then $(G,B)\mapsto f_{hom}(G,B)$ is finite valued and  has quadratic growth. Moreover, it is continuous and satisfies a $2$-Lipschitz property, i.e.,  there is $L>0$ such that for all $G_i\in\R^{3\times 3}$, $B_i\in\R^3$, $i=0,1$ 
		\begin{equation} \label{fhom:2Lip}
			|f_{hom}(G_0,B_0)-f_{hom}(G_1,B_1)|\leq L (1+|G_0|+|G_1|+|B_0|+|B_1|) (|G_0-G_1|+|B_0-B_1|).
		\end{equation} 
	\end{proposition}
	\begin{proof}
		First, note that \eqref{f2} together with the upper estimate of $f_{hom}$  obtained from  \eqref{eq:deffhom} for $k=1$, $\beta=0$,  and $\varphi(z)=-\operatorname{sym}(G_0)z$ if $z\in M$ with $\varphi\in W^{1,2}_0(Q;\R^3)$ implies that  $f_{hom}(G_0,B_0)\le C(1+|G_0|^2+|B_0|^2)$. On the other hand,  for every $\delta>0$ there is $k_0\in\N$, $\varphi_0$, and  $\beta_0$ admissible in \eqref{eq:deffhom}  such that 
        \begin{align*}
         f_{hom}(G_0,B_0)+\delta &\ge \int_Q f(k_0z,G_0+\nabla\varphi_0(k_0z),B_0+\beta_0(k_0z))\, dz\\
         &\ge \frac1C\int_Q (|\operatorname{sym}(G_0+\nabla\varphi_0(k_0z))|^2+|B_0+\beta_0(k_0z)|^2)-C\,d z\\
         &\ge \frac1C(|\operatorname{sym}(G_0)|^2+|B_0|^2)-C
        \end{align*}
        
        in view of 
        \eqref{f1}, \eqref{f2}, and the Jensen inequality.  Consequently, 
        the arbitrareness od $\delta>0$ yields 
        $$
        f_{hom}(G_0,B_0)\ge\frac1C(|\operatorname{sym}(G_0)|^2+|B_0|^2)-C\ .$$
        
		We fix $\Omega=\Lambda=Q$ and $\delta>0$  from  Proposition~\ref{prop:affine} and consider $(u^{\delta,\varepsilon},B^{\delta,\varepsilon})_\varepsilon \rightharpoonup (G_0x,B_0)$ such that \eqref{eq:LinLimsup} holds. Coercivity and boundedness of $f_{hom}$ just proved  imply that 
		$$\frac1C\limsup_{\varepsilon\to 0}(\| \operatorname{sym}(\nabla u^{\delta,\varepsilon})\|^2+\|B^{\delta,\varepsilon}\|^2)-C\le f_{hom}(G_0,B_0)+\delta\le C(1+|G_0|^2+|B_0|^2)+\delta.$$
		
		Define $v^{\delta, \varepsilon}(x)=(G_1-G_0)x +u^{\delta,\varepsilon}(x)$ and $\beta^{\delta,\varepsilon}=B_1-B_0 +B^{\delta,\varepsilon}$.
		Hence, $(v^{\delta, \varepsilon},\beta^{\delta,\varepsilon})\rightharpoonup x\mapsto (G_1x,B_1)$. 
		Additionally, $$\limsup_{\varepsilon\to 0}( \|\operatorname{sym}(\nabla v^{\delta,\varepsilon})\|^2+\|\beta^{\delta,\varepsilon}\|^2)\le C(1+|G_0|^2+|B_0|^2+|G_1|^2+|B_1|^2)+\delta .$$
		We modify $x\mapsto(G_1-G_0)x$ using Lemma~\ref{lem:wlim-approx} to a sequence $(w^\varepsilon)$ such that $\nabla w^\varepsilon=0$ on $ N_{\varepsilon}$, i.e., also   on the inclusion part of the domain $Q$. Moreover, $\nabla w^\varepsilon=G_1-G_0$ on $Q\setminus N_\varepsilon$. 
        Let $\hat w^{\delta,\varepsilon}= w^\varepsilon+ u^{\delta,\varepsilon}$.
		In view of \eqref{eq:LinLiminf} applied to the  sequence  $(\hat w^{\delta, \varepsilon},\beta^{\delta,\varepsilon})_\varepsilon$ and \eqref{f0},\eqref{f3}, and \eqref{f4} we get  
		\begin{align*}
			&f_{hom}(G_1,B_1)-f_{hom}(G_0,B_0)\le \delta+ \liminf_{\varepsilon\to 0}\int_Q f\bigg(\frac x\varepsilon,\nabla \hat w^\varepsilon,\beta^{\delta,\varepsilon}\bigg)-f\bigg(\frac x\varepsilon,\nabla u^{\delta,\varepsilon}, B^{\delta,\varepsilon}\bigg)\,dx\\
            &=\delta+ \liminf_{\varepsilon\to 0}\Bigg[\int_{Q\setminus N_\varepsilon}f\bigg(\frac x\varepsilon,\nabla v^{\delta,\varepsilon},\beta^{\delta,\varepsilon}\bigg) \, dx -\int_{Q\setminus N_\varepsilon} f\bigg(\frac x\varepsilon,\nabla u^{\delta,\varepsilon}, B^{\delta,\varepsilon}\bigg)\,dx\\
           & +\int_{ N_\varepsilon}f\bigg(\frac x\varepsilon,\nabla u^{\delta,\varepsilon},\beta^{\delta,\varepsilon}\bigg) \, dx -\int_{ N_\varepsilon} f\bigg(\frac x\varepsilon,\nabla u^{\delta,\varepsilon}, B^{\delta,\varepsilon}\bigg)\,dx
            \Bigg]\\
			&\le\delta+ \limsup_{\varepsilon\to 0}\int_{Q\setminus N_\varepsilon} C(1+|\operatorname{sym}(\nabla v^{\delta, \varepsilon})|+|\operatorname{sym}(\nabla u^{\delta,\varepsilon})|+|\beta^{\delta,\varepsilon}|+| B^{\delta,\varepsilon}|)(|B_1-B_0|+| G_{1} - G_{0} |)\\
            &+ \limsup_{\varepsilon\to 0}\int_{ N_\varepsilon} C(1+|\operatorname{sym}(\nabla u^{\delta, \varepsilon})|+|\operatorname{sym}(\nabla u^{\delta,\varepsilon})|+|\beta^{\delta,\varepsilon}|+| B^{\delta,\varepsilon}|)(|B_1-B_0|)\\
            &\le\delta + L(1+
            |G_0|+|G_1| +|B_0|+|B_1|)(|B_1-B_0|+| G_{1} - G_{0} |)
		\end{align*} 
		Then we exchange $G_0,B_0$ and $G_1,B_1$ to obtain the estimate because $\delta>0$ was arbitrary.
		
	\end{proof}
	
	%
	\begin{proof}[Proof of Theorem~\ref{thm:ConvLin}]
		We choose a partition of $\Omega$ into Lipschitz subdomains $\{ \Omega_{i} \}_{i}$ and define
		\begin{equation*}
			\lambda_{i}= \frac{1}{|\Omega_{i}|} \int_{\Omega_{i}} \nabla u(x) dx.
		\end{equation*}
		
		For any given $\delta>0$, if the partition $\{ \Omega_{i} \}$ of $\Omega$ is fine enough, we get that 
		\begin{equation} \label{eq:lambdai_close}
			\sum_{i}\int_{\Omega_{i}} \big(|\nabla u -  \lambda_i |^{2}+|B -  \beta_i |^{2}\big) < \delta^2
		\end{equation}
		with suitable constants $\lambda_i\in\R^{3\times 3}$, $\beta_i\in \R^3$ (also depending on $\delta$).
		
		\noindent{\bf (i) $\Gamma-\liminf$ inequality.} 
		We will reduce our general case to that of Proposition~\ref{prop:affine} (i).
		Consider $(u^{\varepsilon},B_{\varepsilon}) \rightharpoonup (u,B)$ weakly in $W^{1,2}(\Omega;\R^3)\times L^2_{\div}(\R^3;\R^3)$ and define 
		\[	
		v^{{\varepsilon},i}:=  \lambda_{i}x - u + u_{\varepsilon}~~~\text{and}~~~
		B^{{\varepsilon},i}:=  \beta_{i} - B + B_{\varepsilon}, \quad x\in \Omega_i.
		\]  
		While $B^{{\varepsilon},i}$ is a perfectly admissible sequence of magnetic fields on $\Omega_i$, divergence free with constant weak limit $\beta_{i}$,
		we cannot directly use $v^{{\varepsilon},i}$ due to the presence of the hard inclusions. 
		However, we may use Lemma \ref{lem:wlim-approx} in a suitable way to compensate this issue. 
		
		Let $\Lambda:=\{x\in \R^3\mid \operatorname{dist}(x;\Omega)<2\}$ and choose an extension of $u$ to $W^{1,2}(\Lambda;\R^3)$.
		Moreover, for each $i$, fix 
		\[
		\text{$\Lambda_i\subset\Lambda$ such that $\Lambda_i\supset \{x\in \R^3\mid \operatorname{dist}(x;\Omega_i)<1\}$.}
		\]
		For each $i$, apply Lemma~\ref{lem:wlim-approx} to $u_i:=\lambda_i x - u$ (instead of $w$) on $\Lambda_i$ (instead of $\Lambda$).
		The lemma yields sequences $\hat{u}_{\eps,i}$ such that $\nabla\hat{u}_{\eps,i}=0$ on $M_{\eps,i}$, i.e., on
		the rigid part $M_\eps$ outside an 
		$\eps$-sized neighborhood of $\partial\Lambda_i$. In particular, 
		$\nabla\hat{u}_{\eps,i}=0$ in $\Omega_i\cap \{\chi_M=1\}$ as long as $\eps$ is small enough. With this, we can modify $v^{{\varepsilon},i}$ to 
		\[
		\hat{v}^{{\varepsilon},i}=\hat{u}_{\eps,i}+u_{\varepsilon}.
		\]
		From \eqref{eq:lambdai_close} and the final estimate in \eqref{eq:hatu}, we also get that
		\begin{equation} \label{eq:hatu_close}
			\sum_{i}\int_{\Omega_{i}} \big(|\nabla \hat{u}_{\eps,i}|^{2}+|\beta_i-B|^{2}\big) < C\delta^2. 
		\end{equation}
		Notice that unlike $v^{{\varepsilon},i}$, $\hat{v}^{{\varepsilon},i}$ now has antisymmetric gradient on $\Omega_i\cap M$,
		while still $\hat{v}^{{\varepsilon},i}\rightharpoonup \lambda_i x$ in $W^{1,2}(\Omega_i;\R^3)$, just like $v^{{\varepsilon},i}$.

		By Proposition \ref{prop:affine}, we have
		\begin{equation} \label{eq:liminf-pwa}
			\liminf_{\varepsilon \rightarrow 0} \sum_{i}  \int_{\Omega_i} f \bigg( \frac{x}{\varepsilon},  \nabla \hat{v}^{\varepsilon,i}, B^{\varepsilon,i} \bigg) \geq \sum_{i} \int_{\Omega_i} f \bigg( \frac{x}{\varepsilon}, \lambda_i, \beta_i \bigg). 
		\end{equation}
		We claim that the left hand side above is close to $\int_{\Omega} f \big( \frac{x}{\varepsilon},\operatorname{sym}(\nabla u^{\varepsilon}),B^{\varepsilon}\big)$. 
		To see this, observe that
		\begin{equation*}
			\begin{aligned}
				&\bigg| \int_{\Omega_i} f \bigg( \frac{x}{\varepsilon},  \nabla u^{\varepsilon},B^{\varepsilon,i} \bigg)
				- \sum_{i} \int_{\Omega_{i}} 
				f \bigg( \frac{x}{\varepsilon},  \nabla \hat{v}^{\varepsilon,i}, B^{\varepsilon} \bigg)\bigg|  
				\\
				& \qquad \begin{aligned}[t] 
					& \leq \sum_{i} \int_{\Omega_{i}} C \big(|\nabla u^{\varepsilon}|+|\nabla \hat{v}^{\varepsilon, i}|+2|B^{\varepsilon}|\big) 
					\big(| \nabla \hat{u}_{i,\varepsilon}(x)|+|\beta_i-B|\big) ds \\
					& \leq C^{\frac{3}{2}} \big(\|u^{\varepsilon}\|_{W^{1,2}}+\|\hat{v}^{\varepsilon, i}\|_{W^{1,2}}+2\|B^{\varepsilon}\|_{L^2}\big)\delta
				\end{aligned}	
			\end{aligned}
		\end{equation*} 
		Here, the last two steps are due to \eqref{f2} (on $\Omega\cap M_\eps$), the Lipschitz property \eqref{f3} of $q$ (on $\Omega\setminus M_\eps$) and \eqref{eq:hatu_close}, using the Cauchy-Schwarz inequality. 
		Since  $f_{hom}$ is also $2$-Lipschitz by Proposition~\ref{prop:fhom}, 
		we can analogously obtain that
		\begin{equation*}
			\begin{aligned}
				&\bigg| \int_{\Omega} f_{hom}(\nabla u,B)
				- \sum_{i} \int_{\Omega_{i}} 
				f_{hom}(\lambda_i, B)\bigg|  
				\\
				& \qquad \begin{aligned}[t] 
					& \leq C^{\frac{3}{2}} \big(\|u^{\varepsilon}\|_{W^{1,2}(\Omega;\R^{3\times 3})}+\sum_{i}\|\lambda_i \|_{L^2(\Omega_i;\R^{3\times 3})}+
					2\|B\|_{L^2(\Omega;\R^{3})}\big)\delta
				\end{aligned}	
			\end{aligned}
		\end{equation*} 
		In total, we infer from \eqref{eq:liminf-pwa} that 
		\begin{equation} \label{eq:thm-lbinner}
			\liminf_{\varepsilon \rightarrow 0} E^\eps (u^\eps,B^\eps;\Omega)\geq E_{hom} (u,B;\Omega)-\tilde{C}\delta.
		\end{equation}
		In addition, we
		have that  
		\begin{equation} \label{eq:thm-lbouter}
			\liminf_{\varepsilon \rightarrow 0} \int_{\R^3\setminus \Omega} f_{ext}(x,B^\eps)\,dx\geq 
			\int_{\R^3\setminus \Omega} f_{ext}(x,B)\,dx=\int_{\R^3\setminus \Omega} f_{ext}(x,B)\,dx.
		\end{equation}
		as a direct consequence of \eqref{fext2}.
		Since $\delta>0$ was arbitrary, \eqref{eq:thm-lbinner} and \eqref{eq:thm-lbouter} combined yield the assertion. 
		
		\noindent{\bf (ii): Recovery sequence.} For every $\delta>0$, choose a piecewise affine approximation $u^\delta$ of $u$ in $W^{1,2}(\Omega;\R^3)$ 
		and a piecewise constant approximation $B^\delta$ of $B$ in $L^2(\Omega;\R^3)$ (possibly with $\div B^\delta\neq 0$ at jumps) such that
		\begin{equation}\label{eq:pa-approx}
			\|u^\delta-u\|_{W^{1,2}(\Omega;\R^3)}+\|B^\delta-B\|_{L^2(\Omega;\R^3)} \leq \delta.
		\end{equation}
		Let $\Omega_{i}=\Omega_{i}(\delta)\subset \Omega$ be an open Lipschitz set (typically a tetrahedron) where $u^\delta$ is affine and 
		$B^\delta$ is constant. Clearly, $\Omega$ can be written as a pairwise disjoint union of such sets, up to a set of measure zero.
		For any $r>0$ (small enough), the sets 
		\begin{equation*}
			\Omega_{i}^{(r)}:=\{x\in \R^3\mid \operatorname{dist}(x;\Omega_i)<r\}
		\end{equation*}
		provide an open covering of $\overline\Omega$, and we can choose a smooth decomposition of unity on $\overline\Omega$
		subordinate to this covering, i.e.,
		\begin{equation*}
			1=\sum_i \varphi_{i,\eps}^{(r)}~~~\text{on}~\overline\Omega,\quad 
			\text{with $\varphi_{i,\eps}^{(r)}\in C_c^\infty(\R^3;[0,1])$ such that $\operatorname{supp}\varphi_{i,\eps}^{(r)}\subset \Omega_{i}^{(r)}$}.
		\end{equation*}	
		Notice that by construction, we also have that 
		\begin{equation*}
			\text{$\varphi_{i,\eps}^{(r)}(x)=1$ for all $x\in \Omega_i$ with $\operatorname{dist}(x;\partial\Omega_i)>r$ }. 
		\end{equation*}	
		In addition, we may assume that there exists a constant $C>1$ independent of $\delta$, $r$ and $\eps$ such that for all $i$,
		\begin{equation} \label{eq:tGC-rec-10}
			|\nabla \varphi_{i,\eps}^{(r)}|\leq \frac{C}{r}\quad\text{and}\quad \nabla \varphi_{i,\eps}^{(r)}(x)=0
            ~\text{on}~M_\eps.
		\end{equation}	
		Here, the second condition above exploits \eqref{Mwellsep}, i.e., that the inclusions are well separated, and forces us to choose $\varphi^{(r)}_{i,\eps}$ dependent on $\eps$, which otherwise would not be necessary.
		
		By Proposition~\ref{prop:affine} (ii) applied with $\Lambda:=\Omega_i^{(1)}:=\{x\mid \operatorname{dist}(x;\Omega_i)<1\}$ instead of $\Omega$, 
		we obtain a $\delta$-almost recovery sequence $(u^{\delta}_{\eps,i},B^{\delta}_{\eps,i})$ for 
		$(u_i^\delta,B_i^\delta)\in W^{1,2}(\Omega_i^{(1)};\R^3)\times L^2_{\div}(\Omega_i^{(1)};\R^3)$ 
		(the natural affine/constant extension of $(u^\delta,B^\delta)$ from $\Omega_i$ to $\Omega_i^{(1)}$): 
		\[
			\limsup_{\eps\to 0} E^\eps (u^{\delta}_{\eps,i},B^{\delta}_{\eps,i};\Omega_i^{(1)})\geq E_{hom} (u^\delta,B;\Omega_i^{(1)}) +
		|\Omega_i^{(1)}|\delta
		\]
		and also  (since the domain of integration $\Omega$ in Proposition~\ref{prop:affine} (ii) can be chosen arbitrarily) 
		\begin{align}\label{eq:recaffine}
			\limsup_{\eps\to 0} E^\eps (u^{\delta}_{\eps,i},B^{\delta}_{\eps,i};\Omega_i) \leq E_{hom} (u^\delta,B^\delta;\Omega_i) + |\Omega_i|\delta.
		\end{align}
		Here, $E^\eps (u,B;\Lambda)$ denotes the restriction of $E^\eps$ to $\Lambda$ as in \eqref{Edomains}, and $E_{hom}(\cdot,\cdot;\Lambda)$ is defined analogously.
		For $0<r\leq 1$, we define
		the functions
		\begin{align*}
			u^{\delta,r}_{\eps}(x):=\sum_i \varphi_{i,\eps}^{(r)}(x) u^{\delta}_{\eps,i}(x),~~x\in \Omega,
		\end{align*}
		and
		\begin{align*}
			B^{\delta,r}_{\eps}(x):=\varphi_{0,\eps}^{(r)}(x) B(x) +\sum_i \varphi_{i,\eps}^{(r)}(x) B^{\delta}_{\eps,i}(x),~~x\in \R^3,
			~~~\text{where}~\varphi_{0,\eps}^{(r)}:=1-\sum_i \varphi_{i,\eps}^{(r)}~\text{on $\R^3$}.
		\end{align*}
		Notice that since $\operatorname{sym}(\nabla\varphi_{i,\eps}^{(r)})=\nabla\varphi_{i,\eps}^{(r)}=0$ on $M_\eps$
		and $\operatorname{sym}(\nabla u^{\delta}_{\eps,i})=0$ on $M_\eps$, 
		we indeed have that $\operatorname{sym}(\nabla u^{\delta,r}_{\eps})=0$ on $M_\eps$ -- the rigidity constraint on the hard inclusions is respected.
		
		We now claim that in the limit as first $\eps\to 0$, then $r\to 0$ and finally $\delta \to 0$,
		the pair $(u^{\delta,r}_{\eps},\cP B^{\delta,r}_{\eps})$, with the divergence-free projection $\cP$ on $\R^3$ of Lemma~\ref{lem:div-free-proj-R3},
		will yield the desired recovery sequence after a suitable diogonalization argument. 
		In fact, it will be easier for us to diagonalize first with $B^{\delta,r}_{\eps}$ instead of $\cP B^{\delta,r}_{\eps}$, and apply the projection $\cP$ afterwards. 	
		The main issue here is to show that for each $\delta>0$, the error created in the transition layers (where $\varphi_{j,\eps}^{(r)}\in (0,1)$ for some $j$) is negligible in the limit as $\eps\to 0$ and then $r\to 0$.
		
		For each $i$, we have that
		\begin{equation} \label{eq:tGC-rec-13}
			\begin{aligned}
				&\|\nabla u^{\delta,r}_{\eps}-\nabla u^{\delta}_{\eps,i}\|_{L^2(\Omega_i;\R^{3\times 3})}\\
				&\qquad
				\leq \Big\|\Big(\sum_{j\neq i} \varphi_{j,\eps}^{(r)}\nabla u^{\delta}_{\eps,j}\Big)+\big(\varphi_{i,\eps}^{(r)}-1\big)\nabla u^{\delta}_{\eps,i}
				\Big\|_{L^2(\Omega_i;\R^{3\times 3})}	
				+\Big\|\sum_j \big(\nabla \varphi_{j,\eps}^{(r)}\big)\otimes u^{\delta}_{\eps,j}\Big\|_{L^2(\Omega_i;\R^{3\times 3})}. 
			\end{aligned}
		\end{equation}
		The two terms on the right hand side of \eqref{eq:tGC-rec-14} will be estimated separately. As to the first, we
		can exploit the equi-integrability of $(|\nabla u^{\delta}_{\eps,j}|^2_\eps$ provided in Proposition~\ref{prop:affine}:
		\begin{equation} \label{eq:tGC-rec-14}
			\begin{aligned}
				&\Big\|\Big(\sum_{j\neq i} \varphi_{j,\eps}^{(r)}\nabla u^{\delta}_{\eps,j}\Big)+\big(\varphi_{i,\eps}^{(r)}-1\big)\nabla u^{\delta}_{\eps,i}
				\Big\|_{L^2(\Omega_i;\R^{3\times 3})}	\\
				&\qquad \leq \sum_j \int_{\{x\in \Omega_i\mid \varphi_{i,\eps}^{(r)}\neq 1\}} |\nabla u^{\delta}_{\eps,j}|^2\,dx
				\leq \sum_j \sup_{0<\eps\leq 1} \int_{\{x\in \Omega_i\mid \operatorname{dist}(x;\partial\Omega_i)<r\}} |\nabla u^{\delta}_{\eps,j}|^2\,dx
				\underset{r\to 0}{\longrightarrow} 0,\\
			\end{aligned}
		\end{equation}
		for every fixed $\delta>0$. As to the second term, we have that
		\begin{equation} \label{eq:tGC-rec-15}
			\begin{aligned}
				& \Big\|\sum_j \big(\nabla \varphi_{j,\eps}^{(r)}\big)\otimes u^{\delta}_{\eps,j}\Big\|_{L^2(\Omega_i;\R^{3\times 3})}\\
				&\qquad \leq 
				\sum_j \Big\|\big(\nabla \varphi_{j,\eps}^{(r)}\big)\otimes (u^{\delta}_{\eps,j}-u^{\delta}_{j})\Big\|_{L^2(\Omega_i;\R^{3\times 3})}
				+\Big\|\sum_j \big(\nabla \varphi_{j,\eps}^{(r)}\big)\otimes u^{\delta}_{j}\Big\|_{L^2(\Omega_i;\R^{3\times 3})} \\
				&\qquad =\sum_j \Big\|\big(\nabla \varphi_{j,\eps}^{(r)}\big)\otimes (u^{\delta}_{\eps,j}-u^{\delta}_{j})\Big\|_{L^2(\Omega_i;\R^{3\times 3})}
				+\Big\|\sum_j \big(\nabla \varphi_{j,\eps}^{(r)}\big)\otimes (u^{\delta}_{j}-u^{\delta}_{i})\Big\|_{L^2(\Omega_i;\R^{3\times 3})},
			\end{aligned}
		\end{equation}
		the latter since $\sum_j  \varphi_{j,\eps}^{(r)}=1$ and therefore
		$
		\sum_j \big(\nabla \varphi_{j,\eps}^{(r)}\big)\cdot u^{\delta}_{i}=\big(\nabla \sum_j  \varphi_{j,\eps}^{(r)}\big) \cdot u^{\delta}_{i}=0. 
		$
		As a piecewise affine function, $u^{\delta}$ is globally Lipschitz with a constant $L=L(\delta)$. 
		If $r$ is chosen small enough (with respect to $\delta$), within $\Omega_i$, $\nabla\varphi_{j,\eps}^{(r)}\neq 0$ is possible only 
		for those $j$ such that $\Omega_j$ is adjacent to or coinciding with $\Omega^i$. Here, recall that by construction of $\varphi_{j,\eps}^{(r)}$, $\{\nabla \varphi_{j,\eps}^{(r)}\neq 0\}\cap \Omega^i$ is contained in an $r$-neighborhood of $\partial\Omega^i$. 
		For any $x\in\Omega_i$ and any pair $(i,j)$ that contributes in the last term of \eqref{eq:tGC-rec-15}, i.e.,
		so that $\nabla \varphi_{j,\eps}^{(r)}\big)\otimes (u^{\delta}_{j}-u^{\delta}_{i})(x)\neq 0$,
		we can thus find a point $\tilde{x}\in \partial\Omega_i\cap \partial\Omega_j$ with $|x-\tilde{x}|<r$ and 
		$u^{\delta}_{j}(\tilde{x})=u^{\delta}_{i}(\tilde{x})=u^\delta(\tilde{x})$. 
		The global Lipschitz property of $u^{\delta}_{j}$ (uniform in $j$) thus gives that
		$|u^{\delta}_{j}(x)-u^{\delta}_{i}(x)|\leq |u^{\delta}_{j}(x)-u^\delta(\tilde{x})|+|u^\delta(\tilde{x})-u^{\delta}_{i}(x)| \leq 2L r$ for all relevant $x$.
		As a consequence, we have that
		\begin{equation*}
			\Big|\big(\nabla \varphi_{j,\eps}^{(r)}\big)\otimes (u^{\delta}_{j}-u^{\delta}_{i})\Big|
			\leq |\nabla \varphi_{j,\eps}^{(r)}| 2L(\delta)r \leq 2CL(\delta)\quad\text{on}~\Omega^i,
		\end{equation*}
		by the aforementioned Lipschitz property and \eqref{eq:tGC-rec-10}. 
		Continuing the estimate of \eqref{eq:tGC-rec-15}, we thus further get that
		\begin{equation} \label{eq:tGC-rec-16}
			\begin{aligned}
				&\Big\|\sum_j \big(\nabla \varphi_{j,\eps}^{(r)}\big)\otimes u^{\delta}_{\eps,j}\Big\|_{L^2(\Omega_i;\R^{3\times 3})}\\
				&\qquad \leq \sum_j \Big\|\big(\nabla \varphi_{j,\eps}^{(r)}\big)\otimes (u^{\delta}_{\eps,j}-u^{\delta}_{j})\Big\|_{L^2(\Omega_i;\R^{3\times 3})}
				+|\{x\in \Omega\mid\nabla \varphi_{j,\eps}^{(r)}\neq 0~\text{for a $j$}\}| 2CL(\delta) \\
				&\qquad \leq \sum_j \Big\|\big(\nabla \varphi_{j,\eps}^{(r)}\big)\otimes (u^{\delta}_{\eps,j}-u^{\delta}_{j})\Big\|_{L^2(\Omega_i;\R^{3\times 3})}
				+C_1(\delta)r2CL(\delta)
			\end{aligned}
		\end{equation}
		Here, we used that $\{x\in \Omega\mid \nabla\varphi_{j,\eps}^{(r)}\neq 0~\text{for a $j$}\}\subset 
		\{x\in \Omega\mid \operatorname{dist}(x;\partial \Omega_j)<r~\text{for a $j$}\}$, the transition region which only depends on $\delta$ and $r$ and is an $r$-neighborhood of a finite union of Lipschitz boundaries with volume bounded by $C_1(\delta)r$, where $C_1=C_1(\delta)$ is a suitable constant. Notice that the first term on the right hand side of \eqref{eq:tGC-rec-16} converges to zero as $\eps\to 0$, by the strong convergence of $u^{\delta}_{\eps,j}$ to 
		$u^{\delta}_{j}$ in $L^2(\Omega_i;\R^{3}$, while the second converges to zero as $r\to 0$, uniformly in $\eps$.
		Altogether, \eqref{eq:tGC-rec-13}, \eqref{eq:tGC-rec-14}, \eqref{eq:tGC-rec-15} and \eqref{eq:tGC-rec-16} yield that
		\begin{equation} \label{eq:tGC-rec-17}
			\begin{aligned}
				&\lim_{r\to 0} \limsup_{\eps\to 0} \|\nabla u^{\delta,r}_{\eps}-\nabla u^{\delta}_{\eps,i}\|_{L^2(\Omega_i;\R^{3\times 3})}=0.
			\end{aligned}
		\end{equation}
		A similar (and much easier) estimate for can be done for $B^{\delta,r}_{\eps}$ instead of $\nabla u^{\delta,r}_{\eps}$, essentially only exploiting the equi-integrability of $(|B^{\delta}_{\eps,i}|^2)_\eps$.
		In this way, we obtain that 
		\begin{equation} \label{eq:tGC-rec-18}
			\begin{aligned}
				&\lim_{r\to 0} \limsup_{\eps\to 0}\|B^{\delta,r}_{\eps}-B^\delta_{\eps,i}\|_{L^2(\Omega_i;\R^{3})}=0.
			\end{aligned}
		\end{equation}
		and
		\begin{equation} \label{eq:tGC-rec-19}
			\begin{aligned}
				&\lim_{r\to 0} \limsup_{\eps\to 0}\|B^{\delta,r}_{\eps}-B\|_{L^2(\R^3\setminus \Omega;\R^{3})}=0.
			\end{aligned}
		\end{equation}

		In view of \eqref{f3} and the fact that $u^{\delta,r}_{\eps}$ respects our rigidity constraint, i.e., $\operatorname{sym}(\nabla u^{\delta,r}_{\eps})=0$ on $M_\eps$, \eqref{eq:recaffine}, \eqref{eq:tGC-rec-17} and \eqref{eq:tGC-rec-18}
		imply that
		\begin{equation} \label{eq:tGC-rec-20}
			\begin{aligned}
				&\limsup_{r\to 0} \limsup_{\eps\to 0} 
				E_\eps(u^{\delta,r}_{\eps},B^{\delta,r}_{\eps};\Omega)-E_{hom}(u^{\delta}, B^{\delta};\Omega) \\
				&\quad \leq 
				\limsup_{r\to 0} \limsup_{\eps\to 0} \Big( \begin{aligned}[t]
					& \big|E_\eps(u^{\delta,r}_{\eps},B^{\delta,r}_{\eps};\Omega)
					-\sum_i E_\eps(u^{\delta}_{\eps,i},B^{\delta,r}_{\eps,i};\Omega_i)\big| \\
					& +\sum_i \big(E_\eps(u^{\delta}_{\eps,i},B^{\delta,r}_{\eps,i};\Omega_i)
					-E_{hom}(u^{\delta}, B^{\delta};\Omega_i)\big) \Big)
				\end{aligned}\\
				&\quad \leq \sum_i |\Omega_i|\delta = |\Omega|\delta.
			\end{aligned}
		\end{equation} 
		In addition, \eqref{eq:tGC-rec-19} and \eqref{fext1} imply that
		\begin{equation} \label{eq:tGC-rec-21}
			\begin{aligned}
				\lim_{r\to 0} \limsup_{\eps\to 0} 
				\big|E_\eps(u^{\delta,r}_{\eps},B^{\delta,r}_{\eps};\R^3\setminus \Omega)-E_\eps(u^{\delta,r}_{\eps},B;\R^3\setminus \Omega)\big|=0.
			\end{aligned}
		\end{equation}
		Combining this with \eqref{eq:tGC-rec-20}, we see that
		\begin{equation} \label{eq:tGC-rec-22}
			\begin{aligned}
				\lim_{\delta\to 0} \limsup_{r\to 0} \limsup_{\eps\to 0} 
				E_\eps(u^{\delta,r}_{\eps},B^{\delta,r}_{\eps})
				=\lim_{\delta\to 0} E_{hom}(u^{\delta}, B^\delta)
				=E_{hom}(u, B).
			\end{aligned}
		\end{equation} 
		Here and below, with a slight abuse of notation, we interpret $B^\delta$ as a 
		function on $\R^3$ by extending it with $B$ on $\R^3\setminus\Omega$, i.e,
		$B^\delta=\chi_\Omega B^{\delta}+\chi_{\R^3\setminus\Omega}B$. 
		For the second equality in \eqref{eq:tGC-rec-22},
		we combined \eqref{eq:pa-approx} and \eqref{eq:tGC-rec-19} with the 
		fact that $E_{hom}(\cdot,\cdot;\Omega)$ is continuous on $W^{1,2}(\Omega;\R^3)\times L^2(\R^3;\R^3)$,
		as a consequence of Proposition~\ref{prop:fhom}.
		
		Finally, we also have that
		\begin{equation}\label{eq:tGC-rec-23}
			\begin{aligned}
				&\lim_{r\to 0}\limsup_{\eps\to 0}\|\div B^{\delta,r}_{\eps}\|_{(\mathring W^{1,2}(\R^3))}\\
				& \begin{aligned}[t]
					&\leq \lim_{r\to 0}\limsup_{\eps\to 0} \|\div (B^{\delta,r}_{\eps}-B^\delta)\|_{(\mathring W^{1,2}(\R^3))^*}+\|\div (B^\delta-B)\|_{(\mathring W^{1,2}(\R^3))^*}\\
					&\leq \lim_{r\to 0}\limsup_{\eps\to 0} \|\div (B^{\delta,r}_{\eps}-B^\delta)\|_{(\mathring W^{1,2}(\R^3))^*}+C\|B^\delta-B\|_{L^2(\R^3;\R^3))}\\
					&=C\|B^\delta-B\|_{L^2(\R^3;\R^3))} \underset{\delta\to 0}{\longrightarrow} 0.
				\end{aligned}
			\end{aligned}
		\end{equation}
		Above, we used that 
		$\|\div (B^{\delta,r}_{\eps}-B^\delta)\|_{(\mathring W^{1,2}(\R^3))^*}\to 0$ as $\eps\to 0$ (first) and $r\to 0$ (second), which can
		can be shown as in the proof of Proposition~\ref{prop:affine} (ii): As essential ingredients, we again have that $B^{\delta,r}_{\eps}-B^\delta\rightharpoonup 0$ in $L^2$, $B^{\delta,r}_{\eps}-B^\delta$ vanishes outside the fixed bounded domain 
		$\tilde\Omega:=\Omega^{(1)}$ (with the $1$-neighborhood $\Omega^{(1)}$ of $\Omega$), and after restriction to $\tilde\Omega$, 
		$L^2(\tilde\Omega)/\R^3$ compactly embeds into
		$(\mathring W^{1,2}(\tilde\Omega))^*\cong (W^{1,2}(\tilde\Omega)/\R^3)^*$. 
		
		A suitable choice of a diagonal sequence $(\tilde{u}_\eps,\tilde{B}^\eps):=(u^{\delta(\eps),r(\eps)}_{\eps},B^{\delta(\eps),r(\eps)}_{\eps})$, with $r(\eps)\to 0$ (slow enough) and $\delta(\eps)\to 0$ (even slower), 
		now gives almost all properties of a recovery sequence: 
		\begin{align}\label{eq:recovery0}
			\begin{aligned}
				&E_\eps (\tilde{u}_\eps,\tilde{B}_\eps)\underset{\eps\to 0}{\longrightarrow}E_{hom} (u,B),\\	
				&\tilde{u}_\eps\underset{\eps\to 0}{\longrightarrow} u~~~
                 \text{strongly in $L^{2}(\Omega;\R^3)$},\quad
                 \tilde{B}_\eps\underset{\eps\to 0}{\rightharpoonup}B~~~
                 \text{weakly in $L^{2}(\Omega;\R^3)$,}\quad\text{and}\\
				& \|\div \tilde{B}^\eps\|_{(\mathring W^{1,2}(\R^3;\R^3))}\underset{\eps\to 0}{\longrightarrow} 0. 
			\end{aligned}
		\end{align} 
        Two asserted properties are still missing, though: $\tilde{B}^\eps$ is not yet divergence-free (but close), and possibly $\tilde{u}_\eps\neq 0$ on $\partial\Omega$. 
		Here, notice that combined with the equi-coercivity of $E_\eps$ (cf.~Remark~\ref{rem:equi-coercive}, in view of the already known strong convergence of $\tilde{u}_\eps$ in $L^2$, we do not need boundary conditions for Korn's inequality here),
        the convergence in the first line of \eqref{eq:recovery0} implies that 
		$(\tilde{u}_\eps,\tilde{B}_\eps)$ is bounded in $W^{1,2}(\Omega;\R^3)\times L^2(\R^3,\R^3)$.
		Hence, we also have that
        \begin{align}\label{eq:recovery0b}
			\begin{aligned}
				&(\tilde{u}_\eps,\tilde{B}_\eps)\underset{\eps\to 0}{\rightharpoonup}B~~~
                 \text{weakly in $W^{1,2}(\Omega;\R^3)\times L^{2}(\Omega;\R^3)$.}
			\end{aligned}
		\end{align}
		To obtain the recovery sequence as asserted, it now suffice to correct the boundary values of $\tilde{u}_\eps$ and replace $\tilde{B}_\eps$ with its divergence-free projection $\cP \tilde{B}_\eps$.

       As to the former, we can again use Lemma~\ref{lem:wlim-approx}: 
        With $\Lambda:=\Omega$ and $w:=\tilde{u}_\eps-u$, we obtain an auxiliary sequence 
        $(\hat{w}_\eps)$, bounded in $W^{1,2}$ such that $\hat{w}_\eps\to 0$ in $L^2$ (since $\tilde{u}_\eps\to u$ in $L^2$, cf.~\eqref{eq:hatwL2conv}) and $\hat{w}_\eps=\tilde{u}_\eps-u$ on the boundary of each $\eps$-cell
        $\eps(z+Q)$, $z\in \mathcal{Z}_\eps$. Choosing a sequence of smooth cut-off functions
        $\varphi_\eps\in C_c^\infty(\Omega;[0,1])$ such that
        \[
            \Omega\setminus\Omega_\eps\subset \{\varphi_\eps=0\},~~~
            \{\varphi_\eps=1\}\nearrow \Omega,~~~\nabla\varphi_\eps=0~~\text{on}~~M_\eps
            ~~~\text{and}~~~|\nabla \varphi_\eps|\hat{w}_\eps\to 0~~\text{in}~~L^2,
        \]
        we obtain the modified displacements
        \[
            \hat{u}_\eps(x):=\begin{cases}
                \tilde{u}_\eps(x)-(1-\varphi_\eps(x))\hat{w}_\eps(x) & \text{if}~x\in \Omega_\eps,\\
                u(x) & \text{if}~x\in \Omega\setminus\Omega_\eps.
            \end{cases}
        \]
        By construction, these belong to $W^{1,2}(\Omega;\R^3)$ (since $\hat{w}_\eps=\tilde{u}_\eps-u$ on $\partial \Omega_\eps$, the piecewise definition does not cause an issue), are bounded in this space and satisfy $\hat{u}_\eps=u=0$ on $\partial\Omega$.
        In addition, $\hat{u}_\eps=\tilde{u}_\eps$ on $\{\varphi_\eps=1\}\nearrow \Omega$ and $\nabla \hat{u}_\eps=\nabla \tilde{u}_\eps$ on $M_\eps$ (so that rigidity in the inclusions holds). Using the Lipschitz property \eqref{f3} to estimate the error in the energy on the asymptotically vanishing boundary layer
        $\{\varphi_\eps\neq 1\}$ (outside of $M_\eps$, where nothing was changed),
        it is now not difficult to see that \eqref{eq:recovery0} and \eqref{eq:recovery0b} also hold with 
        $\hat{u}_\eps$ instead of $\tilde{u}_\eps$.
        
        Finally, we still have to modify $\tilde{B}_\eps$ to a divergence-free field.
        By Lemma~\ref{lem:div-free-proj-R3}, $\|\cP \tilde{B}_\eps-\tilde{B}_\eps\|_{L^2(\R^3;\R^3)}\to 0$,
		which implies that $E_\eps (\tilde{u}_\eps,\tilde{B}_\eps)-E_\eps (\tilde{u}_\eps,\cP\tilde{B}_\eps)\to 0$ by 
		the Lipschitz property \eqref{f3} of $f$ on $\Omega\setminus M_\eps$ and the continuity \eqref{fext1} on its complement. 
		We conclude that
		\begin{align}\label{eq:recovery}
			\begin{aligned}
				&E_\eps (\hat{u}_\eps,\cP\tilde{B}_\eps)\underset{\eps\to 0}{\longrightarrow}E_{hom} (u,B)	~~\text{and}~~\\
				&(\hat{u}_\eps,\cP\tilde{B}_\eps)\underset{\eps\to 0}{\rightharpoonup}(u,B)\quad\text{weakly in $W^{1,2}(\Omega;\R^3)\times L^{2}_{\div}(\Omega;\R^3)$.}
			\end{aligned}
		\end{align}
        The pair $(\hat{u}_\eps,\cP\tilde{B}_\eps)$ thus give the asserted recovery sequence.
	\end{proof}
	
\appendix

\section{Appendix: The energy functional in the large strain case  and alternative models }\label{sec:AppA}

	The energy functional \eqref{Eeps} results from the linearization of the actual energy in a regime of small displacement, small strain, and magnetic field. However, the general magnetoelastic background dictating its properties is better understood in the original, nonlinear regime, which is why we will discuss this here briefly. Even for what is discussed in this appendix, additional assumptions on the elastic part of the energy would be needed to guarantee sufficient invertibility of deformation map and control on the inverse $1/J$ of its Jacobian determinant $J$ to justify the formal change of variables we use to switch between Lagrangian and Eulerian formulations. 
	For simplicity, in this section we discuss the case of a homogeneous body, which is easily extended to an inhomogeneous setting later.

 There are several equivalent formulation for the equilibrium of a magnetoelastic body. These formulations differ for the choice of the intependent variable that describes the magnetization state: the magnetization, the magnetic field, and the magnetic induction. In the present paper, we choose the variational formulation based on the magnetic induction. As discussed in \cite{BuDoOg2008}, the equilibrium configuration of a magnetoelastic body is obtained by minimizing the following functional:
	\begin{equation}
		\begin{aligned}
			E(y,B)=&\int_{\Omega}\Phi(F,B)+\frac1 {2\mu_0}J^{-1}{F}B\cdot FB-\frac1 {2\mu_0} F^\top(b_{\rm a}\circ y)\cdot B\,dx\\
			&+\int_{\mathbb R^3\setminus\Omega}	\frac1 {2\mu_0}J_{\rm a}^{-1}{F_{\rm a}}B\cdot F_{\rm a}B
			-\frac1 {2\mu_0} F_{\rm a}^\top(b_{\rm a}\circ y_{\rm a})\cdot B\,dx,
		\end{aligned}
	\end{equation}
	In the above functional, $\Phi:\R^{3\times 3}\times \R^3\to \R$ (material energy density),
	$b_{\rm a}\in L^2(\R^3;\R^3)$ (applied field) and $\mu_0>0$ (magnetic permeability constant of vacuum) are given, 
	$y:\Omega\to\mathbb R^3$ is the deformation
	\begin{equation}
		F=\nabla y~~~~\text{and}~~~J=\det F=\det \nabla y
	\end{equation}
	are the deformation gradient and the associated Jacobian determinant, and $B$
	is the Lagrangian magnetic field. 
	The function $y_{\rm a}$ is a suitable extension of $y$ from $\Omega$ to $\mathbb R^3$; given a boundary condition $y=y_D$ on $\partial\Omega$ with fixed, globally invertible $y_D$, we can choose $y_{\rm a}=y_D$ independently of $y$.
	Minimization is carried out  over the admissible set
	\begin{equation}
		\mathscr Y=\{(y,B)\in \mathcal Y \times L^2(\R^3;\R^3):y=y_D\text{ on }\partial\Omega,\operatorname{div}B=0\}.
	\end{equation}
	where $\mathcal Y=W^{1,p}(\Omega;\R^3)$ is a suitable Sobolev space.
\begin{remark}
The differential constraint on $B$ could be removed by expressing $B$ as $B=\operatorname{curl}A$. 
Then $\operatorname{div} B=0$ automatically, and given any divergence-free $B\in L^2$,
a suitable vector potential $A\in L^2$ always exists.
Replacing $B$ with $\operatorname{curl} A$, the set of admissible states becomes	
	\begin{equation}
		\tilde{\mathscr Y}=\{(y,A)\in \mathcal Y\times L^2(\R^3;\R^3):y=y_D\text{ on }\partial\Omega
	\end{equation}
However, using $\tilde{\mathscr Y}$ as defined introduces a high degree of nonuniqueness
in the functional, since $B$ does not uniquely determine $A$. This could be remedied 
by again adding a differential constraint, the gauge condition $\operatorname{div}A=0$. 
\end{remark}

Unlike in the small strain regime considered before, the difference between deformed and reference configuration is now relevant for the magnetic induction (the error the difference introduces is of higher order, though, when formally passing to a linearized model).

	The actual, or Eulerian, magnetic induction is given by
	\begin{equation} \label{link_bB}
		b=(J^{-1}FB) \circ \overline y^{-1}
	\end{equation} 
	where $\overline y=y$ in $\Omega$ and $\overline y=y_a$ in $\mathbb R^3\setminus\Omega$. The Eulerian applied magnetic field $b_a$, whose associated energy is, by definition
	\begin{equation*}
		\begin{aligned}
			-\frac{1}{\mu_0}\int_{\mathbb R^3}b_{\rm a}\cdot b dx
			&=-\frac{1}{\mu_0}\int_{\mathbb R^3}J(b_{\rm a}\circ \overline y)\cdot(J^{-1}FB) dx
			=-\frac{1}{\mu_0}\int_{\mathbb R^3}(b_{\rm a}\circ \overline y)\cdot FBdx\\
			&=
			-\frac{1}{\mu_0}\int_{\mathbb R^3}(b_{\rm a}\circ \overline y)\cdot FB dx
			=-\frac{1}{\mu_0}\int_{\mathbb R^3}F^\top(b_{\rm a}\circ \overline y)\cdot Bdx.
		\end{aligned}
	\end{equation*}	
	In this model, 
	\begin{equation}\label{magnetization}
		M=-\partial_B \Phi(F,B),
	\end{equation} 
	is the Lagrangian magnetization. Its correspondent Eulerian quantity is
	\begin{equation*}
		m=F^{-T}M.	
	\end{equation*}
	It follows that the energy $\Phi$ must be independent on $B$ for a non-magnetic material.
The Lagrangian and Eulerian magnetic fields are then defined as
\begin{equation}\label{link_BHM}
	H=\frac 1 {\mu_0 J} F^{\top}F B -M,\qquad  h=\frac b{\mu_0}-m= (F^{-\top}   H) \circ{\overline y}^{-1}.	 	
\end{equation}
 We remark that with these definition of $B$ and $H$ in terms of $y$, $b$ and $h$, $\operatorname{curl} H=0$ if and only if 
 $\operatorname{curl} h=0$, and similarly, $\operatorname{div} B=0$ if and only if $\operatorname{div} b=0$
 as a consequence of \eqref{link_bB} (assuming that $y$ is smoothly invertible).

The corresponding Eulerian formulation formally can be obtained by changing variables $\xi=y(x)$, which yields the functional $E_{E}(y,b)=E(y,B)$ with
	\begin{equation}\label{E(b)}
		\begin{aligned}
			E_{E}(y,b)=\int_{y(\Omega)}\Phi(F\circ y^{-1},b)+\frac1 {2\mu_0}(b-b_a)\cdot b\,d\xi
			&+\int_{\mathbb R^3\setminus y(\Omega)}	\frac1 {2\mu_0} (b-b_a)\cdot b\,d\xi.
		\end{aligned}
	\end{equation}
	with the admissible set
		\begin{equation}
		\mathscr Y_{E}=\{(y,b)\in \mathcal Y \times L^2(\R^3;\R^3):y=y_D\text{ on }\partial\Omega,\operatorname{div}b=0\}.
	\end{equation}
 
\begin{remark}\label{rem:from-MH-to-B} 
A maybe more classical formulation of magnetostatics uses the Eulerian fields $m$ and $h$
instead of $b$;
in that case $b$ is implicitly tracked as $b:=\mu_0(m+h)$ with the magnetic permeability $\mu_0>0$ of vacuum. This formulation is based on an energy functional of the form
\begin{equation*}
		\begin{aligned}
			 \hat{E}_E(y,m,h)=&\int_{y(\Omega)}\hat\Phi(F\circ y^{-1},m)\,d\xi+
   \int_{\mathbb R^3}\frac{\mu_0}{2}|h|^2 \,d\xi 
            -\int_{y(\Omega)} m \cdot b_a \,d\xi
		\end{aligned}
\end{equation*}
subject to the constraints\footnote{$\chi_S$ denotes the 
characteristic function of a set $S$: $\chi_S(x)=1$ if $x\in S$, 
and $\chi_S(x)=0$ otherwise.} 
\begin{equation}\label{mh-constraints}
\begin{aligned}
    \text{$\operatorname{div} (h+\chi_{y(\Omega)}m)=0$, $\operatorname{curl} h=0$,} 
\end{aligned}
\end{equation} 
i.e., the magnetostatic Maxwell equations with the magnetization $m$ effectively only present on the body occupying $y(\Omega)$. Notice that \eqref{mh-constraints} implies that $h=\nabla \varphi$
with the unique solution $\varphi$ of $\Delta \varphi=-\operatorname{div} \chi_{y(\Omega)}m$ on $\R^3$. Hence, $h$ is actually fully determined as a (nonlocal) function of $m$.

The link between $\hat{E}_E$ and $E_{E}$ is not entirely trivial, for details see \cite{BuDoOg2008}. In particular, it is \emph{not true} that $\hat{E}_E(y,m,h)=E_{E}(b)$ for arbitrary $b=\mu_0(h+m)$.
In fact, starting from $\hat{E}_E$, we can only obtain $E_E$ in the form given in \eqref{E(b)} (up to a constant) with $\Phi$ independent of $b_a$ if the externally applied field satisfies $\operatorname{div}b_a=0$ (a natural requirement, though), and then, the state $b$ of 
$E_{E}$ satisfies $b=b_a+\mu_0(h+m)$ at any critical points of $\hat{E}_E$ or $E_{E}$ -- it is the net magnetic induction, with the applied field $-b_a$ removed from the full $\tilde{b}=\mu_0(h+m)$. Moreover, the pointwise transformation between $b$ and $(m,h)$ is given by 
\[
	b+\chi_{y(\Omega)} b_a
	=\frac{\partial}{\partial m}\hat\Phi(F\circ y^{-1},m)+\mu_0 m+\mu_0\chi_{\R^3\setminus y(\Omega)}h
\]
(which also entails \eqref{magnetization}). It, too, can only be expected to hold at critical points: among other things, at other points it in general does not even respect the constraint $\operatorname{div}b=0$ imposed in $E_{E}$. 
Given a density $\hat\Phi$ in $\hat{E}_E$ which is convex in $m$, the transformation allows us to compute the associated density 
in $E_{E}$ as
\begin{align}\label{hatPhi-to-Phi}
	\Phi(\bar F,\bar b)= -\Big(\hat\Phi(\bar F,\cdot)+\frac{\mu_0}{2}|\cdot|^2\Big)^*\big(\bar b\big),~~~\bar b\in \R^3,~\bar F\in \R^{3\times 3}.
\end{align}
Here, $(\cdot)^*$ denotes the Legendre-Fenchel conjugate, i.e., for any function $\R^3\ni \bar m\mapsto g(\bar m)\in \R$,
$g^*(\bar b)=\sup_{\bar m\in \R^3} \bar b\cdot \bar m - g(\bar m)$, $\bar b\in \R^3$. 
Conversely, given a $\Phi$ in $E_E$ which is concave in $b$, 
the magnetization at equilibrium can be recovered as
\begin{align}\label{b-to-m}
	m=-\frac{\partial}{\partial b}\Phi(F,b)
\end{align}
and we can compute $\hat\Phi$ as
\begin{align}\label{Phi-to-hatPhi}
	\hat\Phi(\bar F,\bar m)= \big(-\Phi(\bar F,\cdot)\big)^*(\bar m)-\frac{\mu_0}{2}|\bar m|^2,~~~\bar m\in \R^3,~\bar F\in \R^{3\times 3}.
\end{align}
\end{remark}

 \begin{remark}
	Since $H=\frac 1{\mu_0 J}F^TFB+\partial_B\Phi(F,B)$, if we introduce the potential 
	\begin{equation}
		\Psi(F,B)=\Phi(F,B)+\frac{1}{2 \mu_0} J^{-1} F B \cdot F B,
	\end{equation}
	then we have
	\begin{equation}\label{H-from-B}
		H=\partial_B\Psi(F,B).	
	\end{equation}
	Based on this observation, it would be possible to formulate a dual problem using $H$ instead of $B$ as the state variable, by introducing the Fenchel conjugate $\Psi^*(F,H):=(\Psi(F,\cdot))^*(H)$ of $\Psi$ in its second variable.
 
	However, 
 the stationary solution of the resulting functional would be a saddle point and not a minimum, which is not suited for the application of $\Gamma$-convergence.
 \end{remark}

\section{Appendix: Some properties of the Legendre-Fenchel conjugate}

In general, we are not able to compute Legendre-Fenchel transforms explicitly, but they naturally appear 
in when we make passage from a model in the magnetic state $(m,h)$ to one only depending on $b$, see in particular Example~\ref{ex:mel3}. 
The relevant Legendre-Fenchel transform is carried out in $m$ or $b$, while
the deformation $y$ acts as a fixed parameter. Lacking an explicit expression, it is helpful to at least understand how growth, coercivity and regularity properties are transferred, to be sure that our homogenization results still apply. 
While it possible that the result below is known, we were not able to find anything treating the dependence on parameters to the extent we need here. 

\begin{lemma}[Global Lipschitz properties of the Fenchel transform in presence of parameters] \label{lem:Fenchel-Lipschitz}
Let $\Theta:\R^k\times \R^d\to \R$, $(G,M)\mapsto \Theta(G,M)$, a locally Lipschitz continuous (and thus a.e.~differentiable) function, and consider the Fenchel transform of $\Theta$ in $M$ with parameter $G$, i.e.,
\[
	\Theta^*(G,B):=\sup_{M\in \R^d} (B\cdot M-\Theta(G,M)),~~~B\in \R^d.
\]
\begin{enumerate}
\item[(i)] Suppose that
\begin{align}  
	c\big(|G|^2+|M|^2)-C\leq \Theta(G,M) &\leq C\big(|G|^2+|M|^{p}+1\big)\label{Theta-gc} 
\end{align}
with constants $p\geq 2$ and $C,c>0$ independent of $(G,M)$.
Then 
\begin{align}  
	 c_1|M|^{p'}-C|G|^2-C\leq 
	\Theta^*(G,B) &\leq C_1|M|^{2}-c|G|^2+C\label{Fenchel-gc},
\end{align}
where $C_1,c_1>0$ are constants only depending on $c$, $C$ and $p$, and $\frac{1}{p'}+\frac{1}{p}=1$.
Moreover,
for all $G$, $\Theta^*$
is locally Lipschitz in $B$ and
\begin{align} \label{Fenchel-2Lip-B}
	|D_B\Theta^*(G,B)|\leq L^*(|G|+|B|+1)
\end{align}
with a constant $L^*\geq 0$ independent of $G$ and $B$. 
\item[(ii)] 
Suppose that for a.e.~$(G,M)\in \R^k\times \R^d$, 
\begin{align}  
  |D_G\Theta(G,M)| &\leq L(|G|+|M|^{q}+1), \label{Theta-Lip1}~~\text{and} \\
	|D_M\Theta(G,M)| &\geq c |M|^{p-1}-C\big(|G|+|G||M|^{q-1}+1\big), \label{Theta-gro2}
\end{align}
with constants $p\geq 2$, $1\leq q \leq p-1$ and $L_1,c,C>0$ independent of $(G,M)$.
Then for all $B$, $\Theta^*$
is locally Lipschitz in $G$ and
\begin{align} \label{Fenchel-2Lip-G}
	|D_G\Theta^*(G,B)|\leq L^*(|G|^{\frac{q}{p-q}}+|B|+1)
\end{align}
with a constant $L^*\geq 0$ independent of $G$ and $B$.
\end{enumerate}
\end{lemma}

\begin{remark}\label{rem:Fenchel}
The protoype example for $\Theta$ in Lemma~\ref{lem:Fenchel-Lipschitz} is the nonconvex function
\[
	\Theta(G,M)=|M|^{p}- |M|^{q}|G|+|G|^2.
\]
Of course, in higher dimension $k,d>1$ more complicated expressions with the same homogeneity properties are also admissible, in particular the function $\Theta(G,M):=\hat{\Psi}_G(M)$ of 
Example~\ref{ex:mel3} with $q=\frac{p}{2}$. The case $q=\frac{p}{2}$ is special because 
this is the biggest possible value of $q$ for which our example is still compatible with assumption \eqref{Theta-gc} in (i), and
in (ii), we then have that $\frac{q}{p-q}=1$ in \eqref{Fenchel-2Lip-G}.
\end{remark}
\begin{proof}[Proof of Lemma~\ref{lem:Fenchel-Lipschitz}]
Recall that locally Lipschitz continuous functions are a.e.~differentiable, and any locally uniform bound on their derivatives (where they exist) implies a corresponding Lipschitz property with the same constant. We will only show \eqref{Fenchel-2Lip-B} and \eqref{Fenchel-2Lip-G} assuming that $D\Theta^*$ exists a.e.;
using finite differences instead of derivatives, essentially the same argument can also be used to prove that $\Theta^*$ is locally Lipschitz to justify the use of derivatives. For general information on Fenchel tranforms and some of its basic properties used below we refer to \cite{Ro70B}.

\noindent{\bf (i) Growth, coercivity and Lipschitz properties in $B$.} 
This is essentially well known. 
For each $G$, 
we can apply the Fenchel transform in $M$ to all three expressions compared in the inequality \eqref{Theta-gc}, which reverses the order of the inequalities and changes the sign of additive constants (i.e., independent of $M$) in the terms.
Moreover, up to multiplicative positive constants, $(|\cdot|^2)^*=|\cdot|^2$ and $(|\cdot|^p)^*=|\cdot|^{p'}$. We thus immediately get \eqref{Fenchel-gc}.
For the proof of \eqref{Fenchel-2Lip-B}, recall that $\Theta^*(G,\cdot)$ is also convex for every $G$. In addition, \eqref{Fenchel-gc} implies the quadratic growth condition 
\begin{align}  \label{Fenchel-g}
	|\Theta^*(G,B)| &\leq C_1|B|^{2}+C|G|^2+C
\end{align}
Finally, we recall that any convex function with quadratic growth as in \eqref{Fenchel-g}
automatically satisfies a $2$-Lipschitz property.
We provide a short proof of this for the convenience of the reader, also to emphasize that the growth in $G$ here is indeed as asserted in \eqref{Fenchel-2Lip-B}.
First notice that all tangents of $\Theta^*$ 
with slope $M_0\in \partial \Theta^*(G,B_0)$ at some $B_0$
must satisfy
\[
  \Theta^*(G,B_0)+M_0\cdot (B-B_0)\leq C_1|B|^{2}+C|G|^2+C~~~\text{for all $B\in \R^d$,}
\]
since the convexity of $\Theta^*$ would otherwise immediately contradict \eqref{Fenchel-g}.
In addition, we have that $\Theta^*(G,B_0)+\Theta(G,M_0)\geq M_0\cdot B_0$, 
and using \eqref{Fenchel-g} to control $\Theta(G,M_0)$, we further get that 
\[
  M_0\cdot B\leq C_1(|B|^{2}+|B_0|^{2})+2C|G|^2+2C~~~\text{for all $B\in \R^d$.}
\]
Choosing $B=(|G|+|B_0|+1)\frac{M_0}{|M_0|}$,
we see that for a suitable constant $C_2>0$,
\begin{align}\label{Fenchel-2Lip-B-aux}
  |M_0|(|G|+|B_0|+1)\leq C_2 (|B_0|+|G|+1)^2.
\end{align}
Since $M_0$ was an arbitrary element of $\partial \Theta^*(G,B_0)$ and $\Theta^*(G,\cdot)$ is convex, we in particular have that $M_0=D_B\Theta^*(G,B_0)$ whenever the derivative exists. With this, \eqref{Fenchel-2Lip-B-aux} implies \eqref{Fenchel-2Lip-B} with $L^*=C_2$.

\noindent{\bf (ii) The Lipschitz property \eqref{Fenchel-2Lip-G} of $\Theta^*$ in the parameter $G$.}
Our strategy for the proof of \eqref{Fenchel-2Lip-G} is as follows:
We first show that for every $(G,B)$, we can find a bounded set $K=K(G,B)$ of suitably controlled size such that
\begin{align}\label{Fenchel-choiceK}
	\Theta^*(G,B)=\sup_{M\in K(G,B)} (B\cdot M- \Theta(G,M)). 
\end{align}
If we know that $K(G,B)\subset \hat K=\hat K(G_0,B)$ for all $G$ in a small neighborhood of a point $G_0\in \R^k$,
this also allows us to estimate difference quotients of $\Theta^*(G,B)$: For any $G_1$ with $|G_1-G_0|<1$, by the definition of $\Theta^*$ as a supremum 
(which is always a maximum in compact sets by continuity of $\Theta$): We can choose $M_1\in \overline{K(G_0,B)}$ optimal at $G_1$ so that 
$\Theta^*(G_1,B)=B\cdot M_1-\Theta(G_1,M_1)$
while $\Theta^*(G_0,B)\geq B\cdot M_1-\Theta(G_0,M_1)$.
Hence,
\[
	\Theta^*(G_1,B)-\Theta^*(G_0,B)\leq B\cdot M_1-\Theta(G_1,M_1)-B\cdot M_1+\Theta(G_0,M_1)
	=\Theta(G_0,M_1)-\Theta(G_1,M_1).
\]
An analogous choice of $M_0\in \overline{K(G_0,B)}$ optimal for $\Theta^*$ at $G_0$ yields
\[
	\Theta^*(G_1,B)-\Theta^*(G_0,B)\geq B\cdot M_0-\Theta(G_1,M_0)-B\cdot M_0+\Theta(G_0,M_0)
	=\Theta(G_0,M_0)-\Theta(G_1,M_0).
\]
Combined, this immediately implies 
the estimate
\begin{align} \label{Fenchel-2Lip-3}
	|D_G\Theta^*(G_0,B)| 
	\leq \sup_{M\in \hat K(G_0,B)} |D_G \Theta(G_0,M)|,
\end{align}
which will yield \eqref{Fenchel-2Lip-G} using available bounds for $M\in \hat K(G_0,B)$ implicit in the set.
 
Due to \eqref{Theta-gro2}, we have that
\[
\begin{aligned}
	K(G,B):=\big\{M\in \R^d\mid 
	c|M|^{p-1}-C|G|-C|G||M|^{q-1}-C
	\leq |B|\big\} & \\
	\supset 	\big\{M\in \R^d\mid |D_M\Theta(G,M)|\leq |B|\big\}&.
\end{aligned}
\]
Therefore, $K(G,B)$ is dense in
$\{M\in \R^d\mid \exists S\in \partial_M \Theta(G,M):~|S|\leq |B| \}$,
with $\partial_M \Theta(G,M)\subset \R^d$ denoting the convex subdifferential in the variable $M$, and we indeed obtain
\eqref{Fenchel-choiceK}, because the supremum in the definition of $\Theta^*$ is obtained exactly
at a point $M$ with $B\in \partial_M \Theta(G,M)$. Let
\[
\begin{aligned}
	\hat{K}(G_0,B):= & \big\{M\in \R^d\mid c_1|M|^{p-1}-|G_0|-|G_0||M|^{q-1} \leq |B|+1\big\}\\
	&\supset
	\bigcup \big\{ K(G,B) : G\in\R^d,|G-G_0|<1 \big\}
\end{aligned}
\]
for a suitable constant $0<c_1=c_1(p,q,c,C)<1$ sufficiently small. This choice 
of $c_1$ is indeed always possible,
for instance one can take $c_1\leq C^2$ small enough  
such that
$\big(\frac{c}{C}-\sqrt{c_1}\big)|M|^{p-1}-|M|^{q-1}+\frac{1}{\sqrt{c_1}}-1\geq 0$ for all $M$, 
exploiting that $p>q$. 
This choice implies the set inclusion claimed above: for each $G$ with $|G-G_0|<1$ and all $M\in K(G,B)$,
\[
\begin{aligned}
	&c_1|M|^{p-1}-|G_0|-|G_0||M|^{q-1} \leq c_1|M|^{p-1}-|G|+1-(|G|+1)|M|^{q-1} \\
	&\qquad\leq 
\sqrt{c_1}\Big(\frac{c}{C}|M|^{p-1}-\frac{1}{\sqrt{c_1}}(|G|+|G||M|^{q-1})-|M|^{q-1}+\frac{1}{\sqrt{c_1}}-1\Big) \\
&\qquad\leq 
\sqrt{c_1}\Big(\frac{c}{C}|M|^{p-1}-|G|+|G||M|^{q-1}-|M|^{q-1}+\frac{1}{\sqrt{c_1}}\Big)
\leq \frac{\sqrt{c_1}}{C}|B|+1\leq |B|+1.
\end{aligned}
\]
Our concrete choice of $\hat{K}(G_0,B)$ thus gives us \eqref{Fenchel-2Lip-3}, where we now have to estimate the right hand side using \eqref{Theta-Lip1}.
We distinguish three cases. If $|M|<1$, since $q\geq 1$, we trivially have that
\begin{align}	\label{DGth-est0}
	|D_G \Theta(G_0,M)|\leq L(|M|^q+|G_0|+1)\leq L_1(|M|+|G_0|+1).
\end{align}
If $\frac{c_1}{2}|M|^{p-1}\geq 
|G_0||M|^{q-1}$ and $|M|\geq 1$, 
\[
	\frac{c_1}{2}|M|^q\leq \frac{c_1}{2}|M|^{p-1} \leq |B|+|G_0|+1~~~
	\text{for all $M\in \hat{K}(G_0,B)$}.
\]
Hence, \eqref{Theta-Lip1} implies that
\begin{align}	\label{DGth-est1}
	|D_G \Theta(G_0,M)|\leq L(|M|^q+|G_0|+1)\leq L_1\frac{2}{c_1}|B|+L\Big(\frac{2}{c_1}+1\Big)
	(|G_0|+1).
\end{align}
Finally, if $\frac{c_1}{2}|M|^{p-1}< 
|G_0||M|^{p-1}$ and thus $|M|^{p-q}<\frac{2}{c_1}|G_0|$,
we observe that \eqref{Theta-Lip1} gives
\begin{align}	\label{DGth-est2}
	|D_G \Theta(G_0,M)|\leq L(|M|^{q}+|G_0|+1)\leq 
	L\Big(\frac{2}{c_1}|G_0|\Big)^{\frac{q}{p-q}}+L_1(|G_0|+1).
\end{align}
Combining \eqref{DGth-est0}--\eqref{DGth-est2} to cover all cases, we see that
for all $M\in \hat{K}(G_0,B_0)$,
\[
	|D_G \Theta(G_0,M)|\leq L^* (|B_0|+|G_0|^{\frac{q}{p-q}}+1)~~~~\text{with}~L^*:=
	L \max\Big\{1,\Big(\frac{2}{c_1}\Big)^{\frac{q}{p-q}},\Big(\frac{2}{c_1}+1\Big)\Big\},
\]
concluding the proof of \eqref{Fenchel-2Lip-G}.
\end{proof}

\subsection*{Acknowledgements}
The research of SK and MK was supported by the GA \v{C}R grant 23-04766S. Part of this work was carried out during a research visit of SK at the the University of Rome 3 supported by INdAM-GNFM. GT acknowledges support from the Italian Ministry of University and Research through project PRIN 2022NNTZNM DISCOVER, and the “Departments of Excellence” initiative. The  hardware used for typesetting part of the documents was acquired through support from the Rome Technopole Foundation. Support from INdAM-GNFM is also acknowledged.

\bibliographystyle{abbrv}
\bibliography{references}

\end{document}